\documentclass[a4paper]{amsart}
\pdfoutput=1
\usepackage[pdfauthor={Alexis Marchand},
            pdftitle={Isometric embeddings of surfaces for scl},
            colorlinks=true]{hyperref}
\usepackage[utf8]{inputenc}
\usepackage[T1]{fontenc}
\usepackage[english]{babel}
\usepackage{lmodern}
\usepackage[dvipsnames]{xcolor}
\usepackage{graphicx}
\usepackage{enumitem}
\usepackage{amsmath,amsthm,amssymb,mathrsfs,stmaryrd,eqnarray,relsize}
\usepackage[msc-links]{amsrefs}
\usepackage{caption,subcaption}

\usepackage{tikz}
\usepackage{tikz-3dplot}
\usetikzlibrary{tqft}
\usetikzlibrary{shapes,arrows,decorations.markings}
\usetikzlibrary[patterns,patterns.meta]

\theoremstyle{plain}
\newtheorem{prop}{Proposition}
\numberwithin{prop}{section}
\newtheorem{thm}[prop]{Theorem}
\newtheorem{cor}[prop]{Corollary}
\newtheorem{lmm}[prop]{Lemma}

\newtheorem{innermthm}{Theorem}
\newenvironment{mthm}[1]
  {\renewcommand\theinnermthm{#1}\innermthm}
  {\endinnermthm}
  
  \newtheorem{innermcor}{Corollary}
\newenvironment{mcor}[1]
  {\renewcommand\theinnermcor{#1}\innermcor}
  {\endinnermcor}

\theoremstyle{definition}
\newtheorem{rmk}[prop]{Remark}

\newtheorem{ex}[prop]{Example}
\newtheorem{defi}[prop]{Definition}

\newtheorem{claim}{Claim}

\numberwithin{equation}{section}

\newcommand{\normm}[1]{\ensuremath{\left\|#1\right\|}}
\newcommand{\gromnorm}[1]{\ensuremath{\normm{#1}_{\textnormal{Grom}}}}

\renewcommand{\emptyset}{\varnothing}

\DeclareMathOperator{\cl}{cl}
\DeclareMathOperator{\scl}{scl}

\DeclareMathOperator{\Imm}{Im}

\DeclareMathOperator{\area}{area}
\DeclareMathOperator{\rot}{rot}
\DeclareMathOperator{\Lk}{Lk}
\DeclareMathOperator{\rk}{rk}

\title{Isometric embeddings of surfaces for scl}
\author{Alexis Marchand}
\date{\today}
\address{DPMMS, Centre for Mathematical Sciences, Wilberforce Road, Cambridge CB3 0WB, United Kingdom}
\email{\href{mailto:aptm3@cam.ac.uk}{aptm3@cam.ac.uk}}

\begin{document}

\begin{abstract}
    Let $\varphi:F_1\to F_2$ be an injective morphism of free groups.
    If $\varphi$ is geometric (i.e. induced by an inclusion of oriented compact connected surfaces with nonempty boundary), then we show that $\varphi$ is an isometric embedding for stable commutator length.
    More generally, we show that if $T$ is a subsurface of an oriented compact (possibly closed) connected surface $S$, and $c$ is an integral $1$-chain on $\pi_1T$, then there is an isometric embedding $H_2(T,c)\to H_2(S,c)$ for the relative Gromov seminorm.
    Those statements are proved by finding an appropriate standard form for admissible surfaces and showing that, under the right homology vanishing conditions, such an admissible surface in $S$ for a chain in $T$ is in fact an admissible surface in $T$.
\end{abstract}

\maketitle

\section{Introduction}

    Stable commutator length is a function that measures the homological complexity of elements in a group. 
    In a topological space $X$, an element $w\in\pi_1X$ being homologically trivial means that a loop representing $w$ bounds a surface in $X$.
    Stable commutator length measures the minimal complexity of such a surface, in a stable sense.
    More precisely, if (a power of) $w$ is homologically trivial in $\pi_1X$, the \emph{stable commutator length} of $w$ is
    \[
        \scl_{\pi_1X}(w)=\inf_{f,\Sigma}\frac{-\chi^-(\Sigma)}{2n(\Sigma)},
    \]
    where the infimum is over all maps $f:\Sigma\rightarrow X$ from surfaces which send $\partial\Sigma$ to $w^{n(\Sigma)}$ for some $n(\Sigma)\in\mathbb{N}_{\geq1}$, and $\chi^-(\Sigma)$ denotes the \emph{reduced Euler characteristic} of $\Sigma$ (i.e. the Euler characteristic after discarding disc and sphere components). 
    Such maps $f$ are called \emph{admissible surfaces}.
    We also set $\scl_{\pi_1X}(w)=\infty$ if no power of $w$ is homologically trivial.
    This function $\scl_{\pi_1X}$ is an invariant of the fundamental group $\pi_1X$.
    We give more detailed definitions in \S{}\ref{sec:scl-gromnorm}.
    
    Computations of $\scl$ remain elusive, and there are very few groups in which we can obtain exact values.
    A major exception is the case of free groups, in which Calegari \cite{cal-scl} proved that $\scl$ is computable and has rational values.
    This was later generalised by Chen \cite{chen} to certain graphs of groups (encompassing previous results of Walker \cite{walker}, Calegari \cite{cal-sails}, Chen \cite{chen-freeprods}, Susse \cite{susse} and Clay--Forester--Louwsma \cite{cfl}).
    However, these examples are all, in some sense, one-dimensional.
    A major open question is whether or not $\scl$ is computable and rational in closed surface groups.

    \subsection*{Isometries for scl}
    
    The present paper aims to make progress towards understanding $\scl$ in free and surface groups by focusing on isometries.
    Every group homomorphism $\varphi:G\rightarrow H$ is $\scl$-nonincreasing in the sense that $\scl_G(w)\geq\scl_H\left(\varphi(w)\right)$ for all $w\in G$.
    We are interested in (injective) morphisms that preserve $\scl$, i.e. that satisfy $\scl_G(w)=\scl_H\left(\varphi(w)\right)$ for all $w\in G$ --- those morphisms are called \emph{isometric embeddings}\footnote{There are slight variations as to what different authors mean by isometries for $\scl$, and we try to clarify the terminology in Definition \ref{defi:isom-scl}.} for $\scl$.
    
    Previous isometry results include those of Calegari and Walker \cite{calegari-walker}, who proved that a generic morphism between free groups preserves $\scl$, as well as Chen \cite{chen}, who proved that certain morphisms of graphs of groups are isometric.
    We give more precise statements of their results in Theorem \ref{thm:isom-prev-results}.
    
    There are several reasons why one might be interested in isometries of $\scl$.
    One of them is that this can lead to a better understanding of the structure of the $\scl$-norm on the space $B_1(G;\mathbb{R})$ of real $1$-boundaries on $G$.
    Calegari's Rationality Theorem \cite{cal-scl} implies that, if $G$ is a free group, then the unit ball of the $\scl$ norm is a rational polyhedron, and Calegari \cite{cal-fnb} also proved that certain top dimensional faces of the $\scl$ norm ball are connected to dynamics via the rotation quasimorphism.
    Isometric endomorphisms of $\scl$ in a group $G$ capture the symmetries of the $\scl$ norm ball in $B_1(G;\mathbb{R})$ and could reveal more information about $\scl$ in $G$.
    
    Moreover, given an isometric embedding $\varphi:G\rightarrow H$ from a group $G$ in which we can compute $\scl$ to another group $H$ where $\scl$ is more mysterious, one can use knowledge about $\scl$ in $G$ to learn about $\scl$ in $H$.
    For instance, Chen \cite{chen}*{Corollary 3.12} uses his isometric embedding theorem to show that the $\scl$-spectrum of the Baumslag--Solitar group $BS(m,\ell)$ contains the $\scl$-spectrum of the free product $\mathbb{Z}/m\ast\mathbb{Z}/\ell$, which is better understood.
    Isometric embeddings from surfaces with boundary to closed surfaces are of particular interest to us, as they could potentially allow us to use what we know about $\scl$ in free groups to gain some ground in closed surface groups.
    
    We first consider $\varphi:F_1\rightarrow F_2$ a morphism of free groups.
    There is a condition that one might want to impose on $\varphi$ in order to prove that it is an isometric embedding: if $\scl_{F_1}(w)=\infty$, then it is natural to ask that $\scl_{F_2}\left(\varphi(w)\right)=\infty$.
    This is equivalent to saying that the induced morphism $\varphi_*:H_1\left(F_1\right)\rightarrow H_1\left(F_2\right)$ on abelianisations should be injective.
    An isometric embedding that satisfies this condition will be called a \emph{strong isometric embedding} (see Proposition \ref{prop:isom-strong-isom}).
    Our first result is that, if $\varphi$ is \emph{geometric} --- in the sense that it is induced by an inclusion of surfaces with boundary --- then this condition is sufficient:
    
    \begin{mthm}{\ref{thm:isom-scl}}[Isometric embedding for $\scl$]
        Let $S$ be an oriented, compact, connected surface with nonempty boundary and let $T\subseteq S$ be a subsurface.
        Consider the inclusion-induced map
        \[
            \iota:\pi_1T\rightarrow\pi_1S.
        \]
        If $\iota_*:H_1\left(\pi_1T\right)\hookrightarrow H_1\left(\pi_1S\right)$ is injective, then $\iota$ is a strong isometric embedding for $\scl$.
    \end{mthm}
    
    In fact, it follows from our proof that extremal surfaces are preserved by the isometric embedding of Theorem \ref{thm:isom-scl}.
    This is also the case for the rotation quasimorphism when it is extremal.
    See \S{}\ref{sec:extr} for a more detailed discussion.
    
    \subsection*{Generalisation to closed surfaces}
    
    One might wonder if Theorem \ref{thm:isom-scl} generalises to closed surfaces.
    The answer is negative, as a simple example shows:
    
    \begin{ex}\label{ex:non-isom}
        Consider the inclusion of surfaces $T\hookrightarrow S$ of Figure \ref{fig:ex:non-isom}.
        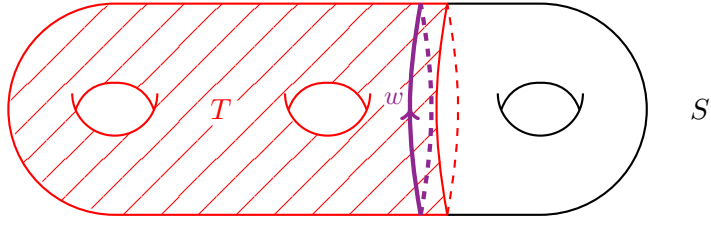
\begin{figure}[htb]
            \centering
            \begin{tikzpicture}[every node/.style={rectangle,draw=none,fill=none},scale=.7]
                \draw [draw=none,,pattern={Lines[angle=45,distance=10pt]},pattern color=red] (-2,0) -- (-1.73,-1) -- (-1,-1.73) -- (0,-2) -- (6.25,-2) -- (6,0) -- (6.25,2) -- (0,2) -- (-1,1.73) -- (-1.73,1) -- cycle;
            
                \draw [red,thick] (6.25,2) to (0,2) (0,2) to[bend right=45] (-2,0) (-2,0) to[bend right=45] (0,-2) (0,-2) to (6.25,-2);
                \draw [thick] (6.25,-2) to (8,-2) (8,-2) to[bend right=45] (10,0) (10,0) to[bend right=45] (8,2) (8,2) to (6.25,2);
                
                \draw [draw=none,fill=white] (0,0) ellipse[x radius=.73cm,y radius=.5cm];
                \draw [red,thick] (-0.73,0) to[bend left=37] (0,.5) (0,.5) to[bend left=37] (0.73,0)
                    (-0.8,.3) to[bend right=45] (0,-.5) (0,-.5) to[bend right=45] (0.8,.3);
                
                \draw [xshift=4cm,draw=none,fill=white] (0,0) ellipse[x radius=.73cm,y radius=.5cm];
                \draw [red,xshift=4cm,thick] (-0.73,0) to[bend left=37] (0,.5) (0,.5) to[bend left=37] (0.73,0)
                    (-0.8,.3) to[bend right=45] (0,-.5) (0,-.5) to[bend right=45] (0.8,.3);
                
                \draw [xshift=8cm,thick] (-0.73,0) to[bend left=37] (0,.5) (0,.5) to[bend left=37] (0.73,0)
                    (-0.8,.3) to[bend right=45] (0,-.5) (0,-.5) to[bend right=45] (0.8,.3);
                
                \draw [thick,red] (6.25,-2) to[bend left=10] (6.25,2);
                \draw [thick,red,dashed] (6.25,-2) to[bend right=10] (6.25,2);
                
                \draw [ultra thick,Plum] (5.75,-2) to[bend left=10] (5.75,2);
                \draw [ultra thick,Plum,dashed] (5.75,-2) to[bend right=10] (5.75,2);
                \draw [ultra thick,Plum,->] (5.55,-.001) -- (5.55,.001);
                
                \draw [draw=none,fill=white] (2,0) circle[radius=.35cm];
                \draw node [red] at (2,0) {\large$T$};
                \draw [draw=none,fill=white] (5.25,.2) circle[radius=.22cm];
                \draw node [Plum] at (5.25,.2) {$w$};
                \draw node at (11,0) {\large$S$};
            \end{tikzpicture}
            \caption{An inclusion of surfaces that is $H_1$-injective but not isometric for $\scl$.}
            \label{fig:ex:non-isom}
        \end{figure}        
        Then the induced map $H_1(T)\rightarrow H_1(S)$ is injective.
        However, let $w\in\pi_1T\hookrightarrow\pi_1S$ be the class of the boundary loop of $T$.
        It is a general fact that an oriented compact surface is extremal for its boundary (see \cite{cal-scl}*{Lemma 4.62}), and therefore
        \[
            \scl_{\pi_1T}(w)=\frac{-\chi^-(T)}{2}=\frac{3}{2}.
        \]
        On the other hand, there is a surface $\Sigma$ of genus $1$ with one boundary component bounding $w$ in $S$, so that
        \[
            \scl_{\pi_1S}(w)\leq\frac{-\chi^-(\Sigma)}{2}=\frac{1}{2}.
        \]
        Therefore, the inclusion-induced map $\pi_1T\hookrightarrow\pi_1S$ does not preserve $\scl$.
    \end{ex}
    
    However, this is not the end of the story.
    Given an admissible surface $\Sigma\rightarrow S$ bounding a power of a given loop $w\in\pi_1T\hookrightarrow\pi_1S$, one can consider the relative homology class represented by $\Sigma$ in $H_2(S,w)$.
    In Example \ref{ex:non-isom}, the two admissible surfaces $T\rightarrow S$ and $\Sigma\rightarrow S$ represent classes in $H_2(S,w)$ that differ by the fundamental class $[S]\in H_2(S)\hookrightarrow H_2(S,w)$.
    Note that this is a phenomenon that cannot occur if $S$ has nonempty boundary, because then $H_2(S)=0$ and all admissible surfaces bounding $w$ projectively represent the same class in $H_2(S,w)$.
    One might therefore ask whether we still get an isometric embedding when the relative homology class is fixed.

    \subsection*{Isometries for the relative Gromov seminorm}
    
    We can make this more precise by introducing the relative Gromov seminorm on $H_2(S,w;\mathbb{Q})$.
    For a topological space $X$, the Gromov seminorm is defined on $H_n(X;\mathbb{Q})$ as the quotient seminorm of the $\ell^1$-norm on the space of $n$-cycles.
    In degree $2$, the Gromov seminorm of a class $\alpha\in H_2(X;\mathbb{Q})$ can also be interpreted as the infimum of $-2\chi^-(\Sigma)/n(\Sigma)$ over maps $\Sigma\rightarrow X$ from closed surfaces representing $n(\Sigma)\alpha$ for some $n(\Sigma)\in\mathbb{N}_{\geq1}$.
    
    Analogously, given $w\in\pi_1X$, we define the \emph{relative Gromov seminorm} of a class $\alpha\in H_2(X,w;\mathbb{Q})$ by
    \[
        \gromnorm{\alpha}=\inf_{f,\Sigma}\frac{-2\chi^-(\Sigma)}{n(\Sigma)}
    \]
    where the infimum is over all maps $f:\Sigma\rightarrow X$ from surfaces with boundary representing $n(\Sigma)\alpha$ for some $n(\Sigma)\in\mathbb{N}_{\geq1}$.
    Hence, $\scl(w)$ can be reinterpreted as the infimum of the relative Gromov seminorm on an affine subspace in $H_2(X,w;\mathbb{Q})$:
    \[
        \scl(w)=\frac{1}{4}\inf_{\substack{\alpha\in H_2(X,w;\mathbb{Q})\\\partial\alpha=\left[S^1\right]}}\gromnorm{\alpha}.
    \]
    See \S{}\ref{subsec:rel-gromnorm} for more details on the relative Gromov seminorm and its relation to $\scl$.
    
    Asking whether an embedding of surfaces $T\hookrightarrow S$ is isometric when a relative homology class is fixed amounts to asking, given $w\in\pi_1T$, whether or not the inclusion-induced map $H_2(T,w)\rightarrow H_2(S,w)$ is isometric for the relative Gromov seminorm.
    
    We answer this question in the affirmative:
    
    \begin{mthm}{\ref{thm:isom-gromnorm}}[Isometric embedding for the relative Gromov seminorm]
        Let $S$ be an oriented, compact, connected surface, let $T\subseteq S$ be a $\pi_1$-injective subsurface, and let $c\in C_1\left(\pi_1T;\mathbb{Z}\right)$ be an integral chain in $T$.
        Then the inclusion-induced map
        \[
            \iota:H_2(T,c;\mathbb{Q})\hookrightarrow H_2(S,c;\mathbb{Q})
        \]
        is an injective isometric embedding for $\gromnorm{\cdot}$.
    \end{mthm}
    
    Note that, if $S$ has nonempty boundary and $w$ is homologically trivial, then $\scl_{\pi_1S}(w)=\frac{1}{4}\gromnorm{\alpha}$ where $\alpha$ is a generator of $H_2(S,w)\cong H_1\left(S^1\right)$.
    Therefore, in this context, Theorem \ref{thm:isom-gromnorm} is really a statement about $\scl$.
    But it has no assumption on $H_1$-injectivity, so it implies a stronger version of Theorem \ref{thm:isom-scl}:
    
    \begin{mcor}{\ref{cor:isom-scl}}
        Let $S$ be an oriented, compact, connected surface with nonempty boundary and let $T\subseteq S$ be a $\pi_1$-injective subsurface.
        Then the inclusion-induced map
        \[
            \iota:\pi_1T\hookrightarrow\pi_1S
        \]
        is an isometric embedding for $\scl$.
    \end{mcor}
    
    One of the points of this paper is to promote the study of the relative Gromov seminorm on $H_2(X,w)$ as an intermediate step in the computation of $\scl(w)$ when $X$ has nontrivial second homology.
    This approach separates the problem of computing $\scl$ into two steps: one can try to understand the relative Gromov seminorm first, and then investigate the infimum in $H_2(X,w)$.
    Hence, some known results about $\scl$ in free groups might generalise to the relative Gromov seminorm in closed surface groups for example, giving partial information about $\scl$ there.
    
    Another instance of this phenomenon is that, even though extremal surfaces are not known to exist for arbitrary elements of closed surface groups, Calegari \cite{cal-fnb}*{Remark 3.18} proved that, if $S$ is a closed surface and $w\in\pi_1S$, then $w$ rationally bounds a positive immersed surface in $S$, and this immersed surface is extremal in its relative homology class.
    In particular, there is a class $\alpha\in H_2(S,w)$ such that $\gromnorm{\alpha}$ is rational.
    
    The hope is that ideas that were successfully applied to the study of $\scl$ in free groups could be used to understand the relative Gromov seminorm in surface groups and wider classes of groups.

    \subsection*{Strategy of proof}
    
    Let $T\subseteq S$ be a subsurface and let $w\in\pi_1T$.
    The general idea to prove Theorems \ref{thm:isom-scl} and \ref{thm:isom-gromnorm} is the following: let $f:\left(\Sigma,\partial\Sigma\right)\rightarrow\left(S,w\right)$ be an admissible surface for $w$ in $S$.
    The goal is to modify $f$ to an admissible surface for $w$ in $T$, as this will show that $\scl_{\pi_1T}(w)\leq\scl_{\pi_1S}(w)$, and the reverse inequality always holds.
    Note that the assumption of Theorem \ref{thm:isom-scl} that $H_1\left(\pi_1T\right)\rightarrow H_1\left(\pi_1S\right)$ be injective is equivalent (in the case where $\partial S\neq\emptyset$) to $H_2(S,T)=0$, which means that every $2$-chain in $S$ with boundary in $T$ does in fact lie in $T$.
    In particular, if $b\in C_2\left(\Sigma\right)$ is a $2$-chain representing the fundamental class $[\Sigma]\in H_2\left(\Sigma,\partial\Sigma\right)$, then this implies that $f_*b\in C_2(T)$.
    If $f$ is an embedding, then we can conclude that $f(\Sigma)\subseteq T$.
    However, this does not follow in general since it might be for example that a $2$-cell $\sigma$ in $S\smallsetminus T$ is not visible in the $2$-chain $f_*b$ because it appears once positively and once negatively, but still $\sigma\subseteq f(\Sigma)$.
    
    We cannot in general assume that admissible surfaces are embedded.
    Our strategy is therefore to find a standard form for admissible surfaces that is general enough to allow one to compute $\scl$ or the relative Gromov seminorm, but nice enough to make the above argument work.
    Our standard form is described in \S{}\ref{sec:std-form} and in particular in Proposition \ref{prop:std-form}; we expect it to be a helpful foundation for further study of $\scl$ in surface groups.
    With this standard form in hand, we can adapt the above argument to show that, under appropriate homology vanishing conditions, any admissible surface in $S$ for $w\in\pi_1T$ is in fact contained in $T$.
    This is the content of Theorem \ref{thm:main}, which implies Theorems \ref{thm:isom-scl} and \ref{thm:isom-gromnorm}.

    \subsection*{Outline of the paper}
    
    In \S{}\ref{sec:scl-gromnorm}, we recall the algebraic and topological definitions of $\scl$, and introduce the relative Gromov seminorm. We then go on to introduce $2$-complexes and discuss some of their topological properties that are relevant for our proof in \S{}\ref{sec:2comp}. In \S{}\ref{sec:std-form}, we show how to reduce admissible surfaces for computations of $\scl$ and $\gromnorm{\cdot}$ in surface groups to a certain standard form. Our main results are proved in \S{}\ref{sec:isom}. Finally, \S{}\ref{sec:extr} is a discussion of how extremal surfaces and quasimorphisms behave with respect to our isometric embeddings.
    
    \subsection*{Acknowledgements}
    
    First and foremost, I would like to thank my supervisor Henry Wilton for our weekly discussions and his constant support.
    Thanks are also due to Lvzhou Chen, Francesco Fournier-Facio, and Kevin Li for helpful conversations, as well as to the anonymous referee for careful reading and detailed comments.
    Financial support from an EPSRC PhD studentship is gratefully acknowledged.

\section{Scl and the relative Gromov seminorm}\label{sec:scl-gromnorm}

    \subsection{Stable commutator length algebraically}
    
    We give the definition of commutator length and stable commutator length via products of commutators. We'll work with $1$-chains throughout this paper, but the reader can harmlessly forget about chains and think about elements of a group.
    
    Let $G$ be a group. A \emph{commutator} in $G$ is an expression of the form $[a,b]=aba^{-1}b^{-1}$, for some $a,b\in G$.
    The \emph{commutator length} $\cl_G(w)$ of an element $w\in G$ is its word length with respect to the set of all commutators:
    \begin{multline*}
        \cl_G(w)=\inf\left\{g\geq1\mid\exists a_1,b_1,\dots,a_g,b_g\in G,\:w=\left[a_1,b_1\right]\cdots\left[a_g,b_g\right]\right\}
        \\\in\mathbb{N}_{\geq0}\cup\{\infty\},
    \end{multline*}
    where we agree that $\inf\emptyset=\infty$.
    
    More generally, given finitely many elements $w_1,\dots,w_k\in G$, we set
    \[
        \cl_G\left(w_1+\cdots+w_k\right)=\inf_{t_i\in G}\cl_G\left(w_1\left(t_1w_2t_1^{-1}\right)\cdots\left(t_{k-1}w_kt_{k-1}^{-1}\right)\right).
    \]
    
    Given $R=\mathbb{Z}$ or $\mathbb{Q}$ or $\mathbb{R}$, we'll denote by $C_n(G;R)$ the group of $n$-chains over $R$, i.e. the free $R$-module with basis $G^n$. These form a chain complex $C_*(G;R)$, called the \emph{bar complex}. We'll write $Z_n(G;R)$ for the group of $n$-cycles and $B_n(G;R)$ for the group of $n$-boundaries. See \cite{weibel}*{Chapter 6} for more details.
    Note that there are natural inclusions $C_n(G;\mathbb{Z})\hookrightarrow C_n(G;\mathbb{Q})\hookrightarrow C_n(G;\mathbb{R})$.
    A chain in $C_n(G;\mathbb{R})$ will be called \emph{real}, a chain in $C_n(G;\mathbb{Q})$ will be called \emph{rational}, and a chain in $C_n(G;\mathbb{Z})$ will be called \emph{integral}.
    
    \begin{defi}
        Given an integral $1$-chain $c=\sum_in_iw_i\in C_1\left(G;\mathbb{Z}\right)$ (with $n_i\in\mathbb{Z}$, $w_i\in G$), the \emph{stable commutator length} of $c$ is defined by
        \[
            \scl_G\left(c\right)=\lim_{m\to\infty}\frac{\cl_G\left(\sum_{i}\left(w_i^{n_i}\right)^m\right)}{m}.
        \]
        The map $\scl_G:C_1\left(G;\mathbb{Z}\right)\rightarrow\left[0,\infty\right]$ is then extended to $C_1\left(G;\mathbb{Q}\right)$ by linearity on rays, and to $C_1\left(G;\mathbb{R}\right)$ by continuity.
    \end{defi}
    
    For more details on why these definitions make sense, we refer the reader to Calegari's book \cite{cal-scl}*{\S{}2.6}.
    
    Observe that, given $c\in C_1\left(G;\mathbb{R}\right)$, we have $\scl_G(c)<\infty$ if and only if $c\in B_1(G;\mathbb{R})$.
    If in addition $H_1\left(G\right)$ is torsion-free, then it is also true that, given $c\in C_1\left(G;\mathbb{Z}\right)$, we have $\scl_G(c)<\infty$ if and only if $c\in B_1(G;\mathbb{Z})$.

    \subsection{Isometric embeddings}
    
    The following is immediate from the definition:
    
    \begin{prop}[Monotonicity of $\scl$]\label{prop:monotonicity-scl}
        Let $\varphi:G\rightarrow H$ be a group homomorphism. Then:
        \begin{enumerate}
            \item For any $w_1,\dots,w_k\in G$, the following inequality holds:
            \[
                \cl_G\left(w_1+\cdots+w_k\right)\geq\cl_H\left(\varphi\left(w_1\right)+\cdots+\varphi\left(w_k\right)\right).
            \]
            \item For any $c\in C_1\left(G;\mathbb{R}\right)$, the following inequality holds:
            \[
                \scl_G\left(c\right)\geq\scl_H\left(\varphi\left(c\right)\right).
            \]
        \end{enumerate}
    \end{prop}
    
    Hence, a group homomorphism is always $\scl$-nonincreasing, and we would like to understand when a group homomorphism preserves $\scl$.
    
    \begin{defi}\label{defi:isom-scl}
        Let $\varphi:G\rightarrow H$ be a group homomorphism.
        \begin{itemize}
            \item We say that $\varphi$ is \emph{$\scl$-preserving} if for every $1$-boundary $c\in B_1(G;\mathbb{R})$, the following equality holds:
            \begin{equation}
                \scl_G\left(c\right)=\scl_H\left(\varphi(c)\right).\label{eq:isom-scl}
            \end{equation}
            \item We say that $\varphi$ is an \emph{isometric embedding} for $\scl$ if it is injective and $\scl$-preserving.
            \item We say that $\varphi$ is a \emph{strong isometric embedding} for $\scl$ if it is injective and $\left(\ref{eq:isom-scl}\right)$ holds for every $1$-chain $c\in C_1(G;\mathbb{R})$.
        \end{itemize}
    \end{defi}
    
    \begin{rmk}
        In Definition \ref{defi:isom-scl}, replacing $\mathbb{R}$ with $\mathbb{Q}$ or $\mathbb{Z}$ does not change what it means for a group homomorphism to be $\scl$-preserving or a (strong) isometric embedding for $\scl$.
    \end{rmk}
    
    It is clear that a strong isometric embedding for $\scl$ is also an isometric embedding since $B_1(G;\mathbb{R})\subseteq C_1(G;\mathbb{R})$. The following clarifies the relation between isometries and strong isometries:
    
    \begin{prop}\label{prop:isom-strong-isom}
        Let $\varphi:G\rightarrow H$ be an isometric embedding for $\scl$. Then $\varphi$ is a strong isometric embedding if and only if the induced map
        \[
            \varphi_*:H_1(G;\mathbb{Q})\rightarrow H_1(H;\mathbb{Q})
        \]
        is injective.
    \end{prop}
    \begin{proof}
        Note that $\varphi$ being a strong isometric embedding means that it preserves the stable commutator length of all chains in $C_1(G;\mathbb{Q})$, not just those in $B_1(G;\mathbb{Q})$.
        Equivalently, for each $c\in C_1(G;\mathbb{Q})$ such that $\varphi(c)$ is a $1$-boundary, $c$ itself is a $1$-boundary.
        But the boundary map $C_1\left(G;\mathbb{Q}\right)\rightarrow C_0\left(G;\mathbb{Q}\right)$ is the zero map, so $C_1\left(G;\mathbb{Q}\right)=Z_1\left(G;\mathbb{Q}\right)$.
        Hence, $\varphi$ is a strong isometric embedding if and only if the preimage under $\varphi:Z_1(G;\mathbb{Q})\rightarrow Z_1(H;\mathbb{Q})$ of $B_1(H;\mathbb{Q})$ is precisely $B_1(G;\mathbb{Q})$, which means that the induced map
        \[
            \varphi_*:H_1(G;\mathbb{Q})\rightarrow H_1(H;\mathbb{Q})
        \]
        is injective.
    \end{proof}
    
    We summarise here some known results about isometries of $\scl$, focusing on free groups:
    \begin{thm}
        \begin{enumerate}
            \item Any left-invertible map $\varphi:G\rightarrow H$ is a strong isometric embedding for $\scl$. (This follows from Proposition \ref{prop:monotonicity-scl}.)
            
            \item \emph{(Calegari \cite{cal-sails}*{Corollary 3.16})} Let $F_m$ and $F_n$ be free groups with respective free bases $\left(a_1,\dots,a_m\right)$ and $\left(b_1,\dots,b_n\right)$, with $m\leq n$, and let $\varphi:F_m\rightarrow F_n$ be given by
            \[
                \varphi:a_i\mapsto b_i^{k_i},
            \]
            with $k_i\in\mathbb{Z}\smallsetminus\{0\}$. Then $\varphi$ is a strong isometric embedding for $\scl$.
            
            \item \emph{(Chen \cite{chen}*{Theorem 3.8})} Let $\mathbb{X},\mathbb{Y}$ be graphs of groups such that $\scl$ vanishes on all vertex groups. Assume that we are given an edge-injective morphism $h:X\rightarrow Y$ between the underlying graphs, monomorphisms $h_v:X_v\hookrightarrow Y_{h(v)}$ between the vertex groups, and isomorphisms $h_e:X_e\xrightarrow{\cong}Y_{h(e)}$ between the edge groups, that commute with the inclusions of the edge groups in the vertex groups, and such that each map $h_v$ induces a morphism that is injective in homology on the sum of the images of the incident edge groups. Then there is an induced map $\underline{h}:\pi_1\mathbb{X}\rightarrow\pi_1\mathbb{Y}$, and $\underline{h}$ is an isometric embedding for $\scl$.
            
            \item \emph{(Calegari--Walker \cite{calegari-walker}*{Theorem 3.16})} Let $F_m$ and $F_n$ be free groups of respective ranks $m$ and $n$. Then there is a constant $C>1$ such that a random homomorphism $\varphi:F_m\rightarrow F_n$ of length $k$ is $\scl$-preserving with probability $1-O\left(C^{-k}\right)$.
        \end{enumerate}
        \label{thm:isom-prev-results}
    \end{thm}

    \subsection{Stable commutator length topologically}\label{subsec:scl-topo}
    
    Stable commutator length can be given a topological interpretation, and we will use this interpretation as a working definition throughout this paper.
    
    Fix a topological space $X$ with $\pi_1X=G$, and let $c=\sum_in_iw_i\in C_1(G;\mathbb{Z})$ be an integral chain (with $n_i\in\mathbb{Z}$, $w_i\in G$).
    We assume that the $w_i$'s are pairwise distinct, so that this decomposition of $c$ is unique.
    For each $i$, we can pick a loop $\gamma_i:S^1\rightarrow X$ representing the conjugacy class of $w_i^{n_i}$ in $X$, where $S^1$ is the (oriented) circle.
    Putting those together, we get a map $\gamma:\coprod_iS^1\rightarrow X$.
    Note that $\gamma$ is uniquely defined up to homotopy.
    
    An \emph{admissible surface}\footnote{Note that, contrary to the standard definition \cite{cal-scl}*{\S{}2.6.1}, we impose no condition on $\partial f_*\left[\partial\Sigma\right]$ at this point. The reason why we are doing this should become clear in \S{}\ref{subsec:rel-gromnorm}: we will want to consider admissible surfaces for all classes in $H_2(X,c)$, not just those mapping to $\textstyle\left[\coprod_iS^1\right]$ under $\partial:H_2(X,c)\rightarrow H_1\left(\coprod_iS^1\right)$.} for $c$ in $X$ is the data of an oriented compact (possibly disconnected) surface $\Sigma$, and of maps $f:\Sigma\rightarrow X$ and $\partial f:\partial\Sigma\rightarrow\coprod_iS^1$ making the following diagram commute:
    \begin{equation}
        \vcenter{\hbox{\begin{tikzpicture}[every node/.style={draw=none,fill=none,rectangle}]
            \node (A) at (0,1.5) {$\partial\Sigma$};
            \node (B) at (2,1.5) {$\Sigma$};
            \node (Ap) at (0,0) {$\coprod_iS^1$};
            \node (Bp) at (2,0) {$X$};
            
            \draw [->] (A) -> (B) node [midway,above] {$\iota$};
            \draw [->] (Ap) -> (Bp) node [midway,above] {$\gamma$};
            \draw [->] (A) -> (Ap) node [midway,left] {$\partial f$};
            \draw [->] (B) -> (Bp) node [midway,left] {$f$};
        \end{tikzpicture}}}
        \label{eq:comm-diag-adm-surf}
    \end{equation}
    where $\iota:\partial\Sigma\hookrightarrow\Sigma$ is the inclusion.
    Such an admissible surface will be denoted by $f:\left(\Sigma,\partial\Sigma\right)\rightarrow\left(X,c\right)$.
    
    The complexity of a compact connected surface $\Sigma$ is measured by its \emph{reduced Euler characteristic} $\chi^-(\Sigma)=\min\left\{0,\chi(\Sigma)\right\}$.
    If $\Sigma$ is disconnected, we set $\chi^-(\Sigma)=\sum_K\chi^-(K)$, where the sum ranges over all connected components $K$ of $\Sigma$.
    
    \begin{prop}[{Calegari \cite{cal-scl}*{Proposition 2.74}}]\label{prop:scl-topo}
        If $X$ is a space with $\pi_1X=G$, and $c\in C_1(G;\mathbb{Z})$ is an integral chain, then there is an equality
        \[
            \scl_G(c)=\inf_{f,\Sigma}\frac{-\chi^-(\Sigma)}{2n(\Sigma)},
        \]
        where the infimum is taken over all admissible surfaces $f:(\Sigma,\partial\Sigma)\rightarrow(X,c)$ such that $\partial f_*\left[\partial\Sigma\right]=n(\Sigma)\left[\coprod_iS^1\right]$ in $H_1\left(\coprod_iS^1;\mathbb{Z}\right)$ for some $n(\Sigma)\in\mathbb{N}_{\geq1}$.
        
        Such an admissible surface will be called an \emph{admissible surface for $\scl_{G}(c)$}.
    \end{prop}
    
    An admissible surface is called \emph{extremal} if it realises the infimum in Proposition \ref{prop:scl-topo}.
    Calegari \cite{cal-sclrat} proved that extremal surfaces exist for all $c\in B_1(G;\mathbb{Z})$ if $G$ if a free group. It follows that $\scl_G$ has rational values in this case.

    \subsection{The Gromov seminorm on homology}
    
    We recall here the definition of the Gromov seminorm, which was introduced in \cite{gromov-vol}.
    
    We'll work with rational coefficients throughout.
    
    Let $X$ be a topological space and let $C_*(X;\mathbb{Q})$ denote its singular chain complex over $\mathbb{Q}$.
    Each module $C_n(X;\mathbb{Q})$ can be equipped with the $\ell^1$-norm $\normm{\cdot}_1$ defined by $\normm{\sum_i\lambda_i\sigma_i}_1=\sum_i\left|\lambda_i\right|$ (with $\lambda_i\in\mathbb{Q}$ and $\sigma_i:\Delta^n\rightarrow X$ a singular $n$-simplex).
    The \emph{Gromov seminorm} (or \emph{$\ell^1$-seminorm}) on $H_n(X;\mathbb{Q})$ is defined to be the quotient seminorm:
    \[
        \gromnorm{\alpha}=\inf_{[a]=\alpha}\normm{a}_1,
    \]
    where the infimum is taken over all $n$-cycles $a\in Z_n(X;\mathbb{Q})$ representing the class $\alpha\in H_n(X;\mathbb{Q})$.
    
    The Gromov seminorm on $H_2(X;\mathbb{Q})$ has the following geometric interpretation (see \cite{m:bavard}*{Proposition 2.7} or \cite{cal-scl}*{\S{}1.2.5} for a proof):
    \begin{prop}\label{prop:grom-norm-geom}
        If $X$ is a topological space, then the Gromov seminorm of $\alpha\in H_2(X;\mathbb{Q})$ is given by
        \[
            \gromnorm{\alpha}=\inf_{f,\Sigma}\frac{-2\chi^-(\Sigma)}{n(\Sigma)},
        \]
        where the infimum is taken over all maps $f:\Sigma\rightarrow X$ from oriented closed surfaces $\Sigma$ such that $f_*[\Sigma]=n(\Sigma)\alpha$ for some $n(\Sigma)\in\mathbb{N}_{\geq1}$.
    \end{prop}

    \subsection{The relative Gromov seminorm}\label{subsec:rel-gromnorm}
    
    The Gromov seminorm is related to $\scl$ via the filling norm as explained in \cite{cal-scl}*{\S{}2.6}, but it is another connection that we would like to discuss here.
    We'll introduce an analogue of $\gromnorm{\cdot}$ on $H_2(X,c)$, where $X$ is a space and $c$ is a chain, and whose calculation is an intermediate step in the calculation of $\scl$.
    
    We first explain what we mean by the homology of a space relative to a chain.
    Consider a topological space $X$ with $\pi_1X=G$ and an integral chain $c\in C_1(G;\mathbb{Z})$; this yields a map $\gamma:\coprod_iS^1\rightarrow X$ as explained in \S{}\ref{subsec:scl-topo}.
    Now let $X_\gamma$ denote the mapping cylinder of $\gamma$:
    \[
        \textstyle X_\gamma=\left(X\amalg\left(\coprod_iS^1\times[0,1]\right)\right)/\sim,
    \]
    where $\sim$ is the equivalence relation generated by $\left(u,0\right)\sim\gamma(u)$ for $u\in\coprod_iS^1$.
    
    There is a natural embedding $\coprod_iS^1\hookrightarrow X_\gamma$ via $u\mapsto(u,1)$, and the homology of the pair $(X,c)$ is defined by
    \[
        \textstyle H_*(X,c;\mathbb{Q})=H_*\left(X_\gamma,\coprod_iS^1;\mathbb{Q}\right).
    \]
    Note that our choice of topological representative $\gamma$ for $c$ was unique up to homotopy (see \S{}\ref{subsec:scl-topo}), and homotopic choices of $\gamma$ will yield homotopic pairs $\left(X_\gamma,\coprod_iS^1\right)$; hence the homology $H_*(X,c)$ only depends on $c$.
    
    We will sometimes omit $\mathbb{Q}$ from the notation, but the relative homology $H_*(X,c)$ should always be understood to be with rational coefficients throughout this paper.
    
    \begin{prop}\label{prop:les-rel-h}
        Let $X$ be a topological space and let $c\in C_1(\pi_1X;\mathbb{Z})$ be an integral chain.
        \begin{enumerate}
            \item There is a long exact sequence\label{prop:les-rel-h-1}
            \[
                \textstyle\cdots\rightarrow H_n\left(\coprod_iS^1\right)\xrightarrow{\gamma_*}H_n\left(X\right)\rightarrow H_n\left(X,c\right)\xrightarrow{\partial}H_{n-1}\left(\coprod_iS^1\right)\rightarrow\cdots.
            \]
            
            \item If $c\in B_1\left(\pi_1X;\mathbb{Q}\right)$, then $\gamma_*\left[\coprod_iS^1\right]=0$.
            
            If in addition $c\in\pi_1 X$ (i.e. $\coprod_iS^1$ consists of a single circle), then $\gamma_*=0$ and there is a short exact sequence\label{prop:les-rel-h-2}
            \[
                \textstyle0\rightarrow H_2(X)\rightarrow H_2(X,c)\xrightarrow{\partial}H_1\left(S^1\right)\rightarrow0.
            \]
            
            \item If $Y\subseteq X$ and $c\in C_1\left(\pi_1Y;\mathbb{Z}\right)$, then there is a long exact sequence\label{prop:les-rel-h-3}
            \[
                \textstyle\cdots\rightarrow H_n\left(Y,c\right)\rightarrow H_n\left(X,c\right)\rightarrow H_n\left(X,Y\right)\xrightarrow{\partial}H_{n-1}\left(Y,c\right)\rightarrow\cdots.
            \]
        \end{enumerate}
        All the above exact sequences are with omitted rational coefficients.
    \end{prop}
    \begin{proof}
        This follows from the long exact sequences of pairs and triples in homology \cite{hatcher}*{p.118}, together with the fact that $X_\gamma$ deformation retracts to $X$ \cite{hatcher}*{p.2}.
    \end{proof}
    
    Observe that, given an admissible surface $f:\left(\Sigma,\partial\Sigma\right)\rightarrow(X,c)$, the commutative square $\left(\ref{eq:comm-diag-adm-surf}\right)$ gives an induced map
    \[
        f_*:H_*(\Sigma,\partial\Sigma)\rightarrow H_*(X,c).
    \]
    In particular, we have a class $f_*[\Sigma]\in H_2(X,c)$, where $[\Sigma]\in H_2\left(\Sigma,\partial\Sigma\right)$ is the (rational) fundamental class of $\Sigma$.
    Proposition \ref{prop:scl-topo} expresses $\scl$ as an infimum over all admissible surfaces with a condition on $\partial f_*\left[\partial\Sigma\right]=\partial\left(f_*[\Sigma]\right)$; we can instead focus on admissible surfaces for which we impose a condition on $f_*[\Sigma]$.
    
    \begin{defi}\label{defi:rel-grom}
        Let $X$ be a topological space and $c\in C_1(\pi_1X;\mathbb{Z})$. The \emph{relative Gromov seminorm} is defined on $H_2(X,c;\mathbb{Q})$ by
        \[
            \gromnorm{\alpha}=\inf_{f,\Sigma}\frac{-2\chi^-(\Sigma)}{n(\Sigma)},
        \]
        where the infimum is taken over all admissible surfaces $f:\left(\Sigma,\partial\Sigma\right)\rightarrow\left(X,c\right)$ such that $f_*[\Sigma]=n(\Sigma)\alpha$ for some $n(\Sigma)\in\mathbb{N}_{\geq1}$.
        
        Such an admissible surface will be called an \emph{admissible surface for $\gromnorm{\alpha}$}.
    \end{defi}
    
    The relative Gromov seminorm as we define it is indeed a seminorm: homogeneity is obtained by replacing an admissible surface with a finite cover, and subadditivity by taking disjoint unions of admissible surfaces (after ensuring that they have the same degree by possibly taking finite covers).
    
    \begin{rmk}
        If $c=0$, then $H_2(X,c)=H_2(X)$, and Proposition \ref{prop:grom-norm-geom} says that $\gromnorm{\cdot}$ coincides with the usual definition of the Gromov seminorm.
        In the general case, the relative Gromov seminorm can also be defined as an $\ell^1$-seminorm --- see \cite{m:bavard}*{\S{}2} for more details.
    \end{rmk}
    
    The connection between $\scl$ and the relative Gromov seminorm is as follows:
    
    \begin{prop}\label{prop:scl-inf-gt}
        Given an integral chain $c\in C_1\left(\pi_1X;\mathbb{Z}\right)$, we have
        \[
            \scl(c)=\frac{1}{4}\inf\left\{\textstyle\gromnorm{\alpha}\mid\alpha\in H_2(X,c),\:\partial\alpha=\left[\coprod_iS^1\right]\right\},
        \]
        where $\partial:H_2(X,c)\rightarrow H_1\left(\coprod_iS^1\right)$ is the boundary map in the long exact sequence of Proposition $\ref{prop:les-rel-h}.\ref{prop:les-rel-h-1}$.
    \end{prop}
    \begin{proof}
        This is a restatement of Proposition \ref{prop:scl-topo}.
    \end{proof}
    
    Proposition \ref{prop:scl-inf-gt} suggests that computations of $\scl$ could be tackled in two successive steps: first fix a relative class $\alpha\in H_2(X,c)$ with $\partial\alpha=\left[\coprod_iS^1\right]$ and estimate $\gromnorm{\alpha}$, and then find the infimum over all such classes $\alpha$.
    Note that, if $H_2(X)=0$ (which happens for instance if $G$ is free and $X$ is a $K(G,1)$), then the long exact sequence of Proposition \ref{prop:les-rel-h} tells us that $\partial:H_2(X,c)\rightarrow H_1\left(\coprod_iS^1\right)$ is injective.
    If in addition $c$ is a boundary, then there is a unique $\alpha\in H_2(X,c)$ such that $\partial\alpha=\left[\coprod_iS^1\right]$, and in this case, $\scl(c)=\frac{1}{4}\gromnorm{\alpha}$ by Proposition \ref{prop:scl-inf-gt}.
    
    However, our point is that in some cases, computations of $\scl$ can be made difficult by the presence of nonzero classes in $H_2(X)$, and in those cases, one may hope to obtain information on the relative Gromov seminorm as a stepping stone towards $\scl$.

\section{Links and orientability in 2-complexes}\label{sec:2comp}

    We start with an analysis of some topological properties of $2$-complexes.
    In particular, we will need a notion of orientability for $2$-complexes that are not surfaces.
    The right setting to make this work will be $2$-complexes with small links.
    The goal of this section is to introduce those notions.
    
    \subsection{2-complexes}\label{subsec:2comp}
    
    We first specify the category of topological spaces we will be working with throughout this paper.
    
    Following the terminology of \cite{fritsch-piccinini}*{Chapter 2}, we say that a continuous map $f:X\rightarrow Y$ between CW-complexes is \emph{cellular} if, for each $n\in\mathbb{N}_{\geq0}$, $f\left(X^{(n)}\right)\subseteq Y^{(n)}$, where $X^{(n)}$ and $Y^{(n)}$ denote the $n$-skeleta of $X$ and $Y$ respectively.
    We say that $f$ is \emph{combinatorial} if it maps each cell of $X$ homeomorphically onto a cell of $Y$.
    
    A \emph{$2$-complex} is a $2$-dimensional CW-complex $X$ such that the attaching map $S^1_\sigma\rightarrow X^{(1)}$ of each $2$-cell $\sigma$ of $X$ is combinatorial for a suitable subdivision of the circle $S^1_\sigma$.
    An edge in this subdivision of $S^1_\sigma$ is called a \emph{side} of $\sigma$, and the \emph{degree} of $\sigma$ is its number of sides.
    This follows Gersten's terminology \cite{gersten}.
    
    We will assume that all $2$-complexes are locally finite, but we allow noncompact $2$-complexes.
    
    Each vertex $v$ in a $2$-complex $X$ has a \emph{link} --- denoted by $\Lk_X(v)$ --- which can be defined as the boundary of a regular neighbourhood of $v$ in $X$ (see \cite{gersten}).
    The link has the structure of a graph, with vertices of $\Lk_X(v)$ corresponding to oriented half-edges ${e}$ of $X$ starting at $v$, with an edge between ${e}_1$ and ${e}_2$ in $\Lk_X(v)$ corresponding to each $2$-cell $\sigma$ of $X$ whose boundary traverses ${e}_1^{-1}$ and ${e}_2$ successively.
    
    A \emph{cellulated surface} is a $2$-complex that is also a topological surface, possibly with boundary.

    \subsection{A surface criterion for 2-complexes}\label{subsec:lks-srf}
    
    To decide whether or not a given $2$-complex is a surface, it suffices to examine the topology of links of vertices; this is the content of the following lemma:
    \begin{lmm}[Surface criterion]\label{lmm:srf-crit}
        A $2$-complex $X$ is a cellulated surface if and only if the link of every vertex in $X$ is a circle or a nondegenerate arc. In this case, a vertex $v$ of $X$ lies on the boundary if and only if its link is homeomorphic to an arc.
    \end{lmm}
    \begin{proof}
        The direct implication $\left(\Rightarrow\right)$ is clear from the definition of the link $\Lk_X(v)$ as the boundary of a regular neighbourhood of $v$.
        For $\left(\Leftarrow\right)$, the key point is that every vertex $v$ has a neighbourhood homeomorphic to a cone over $\Lk_X(v)$.
        If $\Lk_X(v)$ is a circle, then a cone over $\Lk_X(v)$ is homeomorphic to a ($2$-dimensional) disc; if $\Lk_X(v)$ is a nondegenerate arc, then a cone over $\Lk_X(v)$ is homeomorphic to a half-disc.
        It is also clear that points in the interior of $2$-cells have neighbourhoods homeomorphic to $\mathbb{R}^2$.
        It remains to consider points in the interior of edges.
        Given an edge $e$, let $v$ be one of its endpoints; then $\Lk_X(v)$ has a vertex $\hat{e}$ corresponding to $e$, and by assumption, $\hat{e}$ has one or two neighbours in $\Lk_X(v)$.
        Since our $2$-complexes are assumed to be combinatorial, this means that $e$ is incident to one or two $2$-cells of $X$; in both cases, points in the interior of $e$ have neighbourhoods homeomorphic to $\mathbb{R}^2$ or $\mathbb{R}\times\mathbb{R}_{\geq0}$.
    \end{proof}
    
    This motivates the following:
    \begin{defi}\label{defi:small-lks}
        A $2$-complex $X$ has \emph{small links} if one of the following equivalent conditions holds:
        \begin{enumerate}
            \item The link of every vertex of $X$ is homeomorphic to a circle or a union of (possibly degenerate) arcs.
            \item Every edge $e$ of $X$ is incident to at most two $2$-cells, counted with multiplicity (i.e. a $2$-cell is counted as many times as it has sides that are glued to $e$).
        \end{enumerate}
    \end{defi}
    
    In other words, Lemma \ref{lmm:srf-crit} says that a $2$-complex $X$ is a surface if and only if it has nondegenerate small connected links.

    \subsection{Orientability of 2-complexes}\label{subsec:or}
    
    We will need a notion of orientation for $2$-complexes.
    To define it, we will work with \emph{locally finite homology}, denoted by $H_*^\textnormal{lf}$.
    For a CW-complex $X$, this is defined as the homology of the chain complex $C_*^\textnormal{lf,cell}(X)$, where $C_n^\textnormal{lf,cell}(X)$ consists of infinite formal sums of oriented $n$-cells of $X$ with locally finite support.
    Note that, if $X$ is compact, then $H_*^\textnormal{lf}(X)=H_*(X)$.
    See \cite{geoghegan}*{Chapter 11} or \cite{loeh-phd}*{\S{}5.1.1} for more details on locally finite homology.
    
    \begin{defi}\label{defi:boundary}
        Given a $2$-complex $X$, we define the \emph{boundary} $\partial X$ of $X$ to be the $1$-dimensional subcomplex consisting of all the edges of $X$ (and their endpoints) that are incident to a single $2$-cell, and along only one side of this $2$-cell.
        In other words, these are the edges $e$ for which each point in the interior of $e$ has a neighbourhood in $X$ that is homeomorphic to a half-disc.
    \end{defi}
    
    Note that, in a $2$-complex, $C_3^\textnormal{lf,cell}(X)=0$, so $H_2^\textnormal{lf}(X)=Z_2^\textnormal{lf,cell}(X)$.
    In particular, it makes sense to speak of the \emph{support} of a $2$-class: this is just the support of the corresponding $2$-cycle.
    
    \begin{defi}
        Let $X$ be a $2$-complex and let $A$ be an abelian group. We say that $X$ is \emph{$A$-orientable} if there is a class $\beta\in H_2^\textnormal{lf}(X,\partial X;A)$ whose support contains every $2$-cell of $X$.
    \end{defi}
    
    Note that, if $X=S$ is a cellulated surface, then our definition of boundary coincides with the usual one, and $\mathbb{Z}$-orientability of $S$ is equivalent to orientability of $S$ in the usual sense (see for example \cite{vick}*{Corollary 6.7} in the closed case).
    Our definition of orientation applies to any $2$-complex, but in the context of surfaces, it is less intrinsic and flexible than the usual one because it requires one to fix a cellular structure first.
    
    We will use orientability via the following lemma:
    \begin{lmm}\label{lmm:or-h2-zero}
        Let $X$ be an $A$-orientable $2$-complex.
        Consider a subcomplex $Y$ of $X$ such that $\partial X\subseteq Y\subseteq X$ and $H_2^\textnormal{lf}(X,Y;A)=0$.
        Then $Y$ contains every $2$-cell of $X$.
    \end{lmm}
    \begin{proof}
        The long exact sequence of the triple $\left(X,Y,\partial X\right)$ shows that the inclusion induces a surjective morphism
        \[
            H_2^\textnormal{lf}(Y,\partial X;A)\twoheadrightarrow H_2^\textnormal{lf}(X,\partial X;A).
        \]
        Since $X$ is $A$-orientable relative to $\partial X$, there is a class $\beta\in H_2^\textnormal{lf}(X,\partial X;A)$ with support containing every $2$-cell of $X$.
        Let $\beta_0\in H_2^\textnormal{lf}(Y,\partial X;A)$ be a preimage of $\beta$.
        Then the support of $\beta_0$ is contained in $Y$ and must contain every $2$-cell of $X$.
    \end{proof}
    
    \begin{cor}\label{cor:rel-hom-srf}
        Let $S$ be an orientable closed surface and let $T\subseteq S$ be a subsurface such that $H_2(S,T)=0$.
        Then $S=T$.\qed
    \end{cor}

    \subsection{Subcomplex stability}
    
    For our proof, we will need to pass to a subcomplex, and it will be necessary to check that the relevant properties of the original $2$-complex are inherited by the subcomplex.
    We start with the following easy observation:
    
    \begin{lmm}[Subcomplex stability of small links]\label{lmm:subcx-stab-lks}
        Any subcomplex $X_0$ of a $2$-complex $X$ with small links also has small links.
    \end{lmm}
    \begin{proof}
        For each vertex $v\in X_0$, there is an embedding $\Lk_{X_0}(v)\hookrightarrow\Lk_X(v)$, and any subgraph of a circle or a union of arcs is again a circle or a union of arcs.
    \end{proof}
    
    We also need to check that orientability, as well as vanishing of relative homology, descend to subcomplexes:
    
    \begin{lmm}[Subcomplex stability of orientability]\label{lmm:subcx-stab-or}
        Let $A$ be an abelian group and let $X$ be an $A$-orientable $2$-complex with small links.
        Then any subcomplex $X_0\subseteq X$ is $A$-orientable.
    \end{lmm}
    \begin{proof}
        Orientability of $X$ means that there exists a relative cellular $2$-cycle $p=\sum_\sigma\lambda_\sigma\sigma\in Z_2^\textnormal{lf,cell}\left(X,\partial X;A\right)$ (with $\lambda_\sigma\in A$ for each $2$-cell $\sigma$ of $X$) whose support contains all $2$-cells of $X$.
        Set
        \[
            p_0=\sum_{\sigma\subseteq X_0}\lambda_\sigma\sigma.
        \]
        Since $dp\in C_1^\textnormal{cell}\left(\partial X;A\right)$, the support of $dp_0$ consists of $1$-cells of $X$ that lie in $\partial X$ or are incident to at least one $2$-cell in $X\smallsetminus X_0$.
        In both cases, they are incident to at most one $2$-cell of $X_0$; moreover, they are incident to at least one $2$-cell of $X_0$ as they lie in the support of $dp_0$.
        Therefore, the support of $dp_0$ is contained in $\partial X_0$, showing that $p_0$ is a relative $2$-cycle in $Z_2^\textnormal{lf,cell}\left(X_0,\partial X_0;A\right)$ whose support contains all the $2$-cells of $X_0$.
        Hence, $X_0$ is $A$-orientable.
    \end{proof}
    
    \begin{lmm}[Injectivity of relative homology]\label{lmm:subcx-stab-van-fund-class}
        Let $X$ be a $2$-complex, and let $Y,X_0\subseteq X$ be two subcomplexes.
        Set $Y_0=Y\cap X_0$.
        Then for any abelian group $A$, the inclusion-induced map
        \[
            H_2\left(X_0,Y_0;A\right)\rightarrow H_2\left(X,Y;A\right)
        \]
        is injective.
    \end{lmm}
    \begin{proof}
        We follow an argument of Howie \cite{howie}*{Lemma 3.2}.
        Applying excision to the triple $\left(X_0\cup Y,Y,Y\smallsetminus Y_0\right)$ shows that the inclusion induces an isomorphism
        \[
            H_2\left(X_0,Y_0;A\right)\cong H_2\left(X_0\cup Y,Y;A\right).
        \]
        But since $X$ is a $2$-complex, $H_3\left(X,X_0\cup Y;A\right)=0$, so the long exact sequence of the triple $\left(X,X_0\cup Y,Y\right)$ shows that the inclusion induces an embedding
        \[
            H_2\left(X_0\cup Y,Y;A\right)\hookrightarrow H_2\left(X,Y;A\right).
        \]
        This proves that the inclusion-induced map $H_2\left(X_0,Y_0;A\right)\rightarrow H_2\left(X,Y;A\right)$ is injective.
    \end{proof}

\section{Standard form for admissible surfaces}\label{sec:std-form}

    The aim of this section is to reduce admissible surfaces to a certain standard form for the purpose of computing $\scl$ or the relative Gromov seminorm in surface groups.
    This standard form can be thought of as an analogue of Culler's fatgraphs \cite{culler}.

    \subsection{Incompressibility and monotonicity}\label{subsec:incomp-monot}
    
    Let $X$ be a topological space and let $c\in C_1(\pi_1X;\mathbb{Z})$.
    Suppose that we want to compute $\scl(c)$ or $\gromnorm{\alpha}$ for some $\alpha\in H_2(X,c)$, and consider an admissible surface $f:(\Sigma,\partial\Sigma)\rightarrow(X,c)$ with $f_*[\Sigma]=n(\Sigma)\alpha$.
    Recall that we have a commutative diagram:
    \[
        \vcenter{\hbox{\begin{tikzpicture}[every node/.style={draw=none,fill=none,rectangle}]
            \node (A) at (0,1.5) {$\partial\Sigma$};
            \node (B) at (2,1.5) {$\Sigma$};
            \node (Ap) at (0,0) {$\coprod_iS^1$};
            \node (Bp) at (2,0) {$X$};
            
            \draw [->] (A) -> (B) node [midway,above] {$\iota$};
            \draw [->] (Ap) -> (Bp) node [midway,above] {$\gamma$};
            \draw [->] (A) -> (Ap) node [midway,left] {$\partial f$};
            \draw [->] (B) -> (Bp) node [midway,left] {$f$};
        \end{tikzpicture}}}
    \]
    
    Observe first that we can harmlessly remove any disc- or sphere-component of $\Sigma$, since this does not change $\chi^-(\Sigma)$.
    We then say that $\Sigma$ is \emph{disc- and sphere-free}.
    
    Now assume that there is a simple closed curve $\beta$ in $\Sigma$ with null-homotopic image in $X$.
    Then we may cut $\Sigma$ along $\beta$ and glue two discs on the resulting boundary components.
    This makes $-\chi^-(\Sigma)$ decrease without changing $f_*[\Sigma]$, so it improves our estimate of $\scl(c)$ or $\gromnorm{\alpha}$.
    We can therefore always assume that $f$ is \emph{incompressible}: every noncontractible simple closed curve in $\Sigma$ has noncontractible image in $X$.
    
    Let $\partial_j$ be a boundary component of $\Sigma$; hence $\partial f$ sends $\partial_j$ to a component $S^1_{i}$ of $\coprod_iS^1$.
    We say that $f:\left(\Sigma,\partial\Sigma\right)\rightarrow(X,c)$ is \emph{monotone}\footnote{This definition is in general different from the usual definition of a monotone admissible surface (see \cite{cal-scl}*{Definition 2.12}), but the two definitions coincide if $f$ is an admissible surface for $\scl(c)$, or more generally if the coordinates of $\partial\alpha$ in the basis $\left(\left[S^1_i\right]\right)_i$ of $H_1\left(\coprod_iS^1\right)$ all have the same sign. Recall from \S{}\ref{subsec:scl-topo} that each circle in $\coprod_iS^1$ comes with an orientation, and that $\alpha$ might not necessarily map to all circles with positive orientation under $\partial$. In particular, if $\partial\alpha$ has two components of opposite orientations, then there is no monotone admissible surface for $\gromnorm{\alpha}$ in the usual sense.\label{footnote:monotone}} if the sign of the degree of the restriction
    \[
        \partial f_{\left|\partial_j\right.}:\partial_j\rightarrow S^1_i
    \]
    only depends on $i$.
    In other words, two boundary components of $\Sigma$ mapping to the same component of $\coprod_iS^1$ do so with the same orientation.
    
    In the context of $\scl$, it is a classical fact \cite{cal-scl}*{Proposition 2.13} that one can work with monotone admissible surfaces only.
    Our definition of monotonicity allows us to adapt this to the relative Gromov seminorm\footnote{This would be false with the usual definition --- see Footnote \ref{footnote:monotone}.}.
    Calegari's proof \cite{cal-scl}*{Proposition 2.13} works in our context, where one should deal with boundary components of $\Sigma$ mapping to distinct components of $\coprod_iS^1$ separately.
    See \cite{m:phd}*{Lemma II.1.4} for more details.
    
    \begin{lmm}[Monotone admissible surfaces]\label{lmm:monot}
        Fix a class $\alpha\in H_2(X,c)$. Given an admissible surface $f:(\Sigma,\partial\Sigma)\rightarrow(X,c)$ with $f_*[\Sigma]=n(\Sigma)\alpha$, there is a monotone admissible surface $f':\left(\Sigma',\partial\Sigma'\right)\rightarrow\left(X,c\right)$ with $f'_*\left[\Sigma'\right]=n\left(\Sigma'\right)\alpha$ such that
        \[
            \frac{-\chi^-\left(\Sigma'\right)}{n\left(\Sigma'\right)}\leq\frac{-\chi^-(\Sigma)}{n(\Sigma)}.
        \]
    \end{lmm}

    \subsection{Transversality}\label{subsec:transv}
        
    Similarly to Brady, Clay and Forester's proof of the Rationality Theorem \cite{bcf}, we will use the notion of transversality from \cite{brs}*{\S{}VII.2} to obtain a nice decomposition of admissible surfaces.
        
    Let $X$ be a $2$-complex, and let $c\in C_1(\pi_1X;\mathbb{Z})$ be an integral chain.
    Recall from \S{}\ref{subsec:scl-topo} that $c$ can be represented by a map $\gamma:\coprod_iS^1\rightarrow X$.
    We are considering an admissible surface $f:\left(\Sigma,\partial\Sigma\right)\rightarrow(X,c)$ with $f_*[\Sigma]=n(\Sigma)\alpha$.
    We can apply the Transversality Theorem from \cite{brs}*{\S{}VII.2} to ensure that $\gamma:\coprod_iS^1\rightarrow X$ and $f:\Sigma\rightarrow X$ are \emph{transverse}: this means that $\Sigma$ decomposes into subsurfaces mapping to vertices of $X$, $1$-handles (i.e. trivial $I$-bundles over edges of $X$), and discs mapping homeomorphically onto $2$-cells of $X$.
    
    We have seen in \S{}\ref{subsec:incomp-monot} that $f$ may be assumed to be incompressible.
    This implies that each subsurface of $\Sigma$ mapping to a vertex of $X$ is in fact a disc.
    Hence, $\Sigma$ decomposes into the following pieces:
    \begin{itemize}
        \item Discs mapping to vertices of $X$ --- called \emph{vertex discs},
        \item \emph{$1$-handles}, i.e. trivial $I$-bundles over edges of $X$, and
        \item Discs mapping homeomorphically onto $2$-cells of $X$ --- called \emph{cellular discs}.
    \end{itemize}
    We then say that $f:(\Sigma,\partial\Sigma)\rightarrow(X,c)$ is a \emph{transverse incompressible admissible surface}. See Figure \ref{fig:transv-surf}.    
    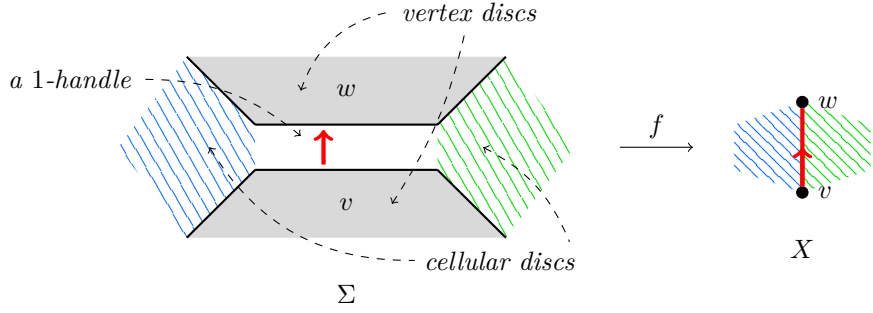
\begin{figure}[hbt]
        \centering
            \centering
            \begin{tikzpicture}[every node/.style={rectangle,draw=none,fill=none},scale=.6]
                \draw [draw=none,pattern={Lines[angle=-60,distance=4pt]},pattern color=NavyBlue] (0,0) -- (-1.5,-1.5) -- (-3,0.5) -- (-1.5,2.5) -- (0,1) -- (0,0);
                \draw [draw=none,pattern={Lines[angle=-60,distance=4pt]},pattern color=ForestGreen] (4,0) -- (5.5,-1.5) -- (7,0.5) -- (5.5,2.5) -- (4,1) -- (4,0);
                \draw [draw=none,fill=gray!30] (0,0) -- (4,0) -- (5.5,-1.5) -- (-1.5,-1.5) -- (0,0);
                \draw [draw=none,fill=gray!30] (0,1) -- (4,1) -- (5.5,2.5) -- (-1.5,2.5) -- (0,1);
                
                \draw [thick] (0,0) -- (4,0) (0,1) -- (4,1)
                    (0,0) -- (-1.5,-1.5) (0,1) -- (-1.5,2.5)
                    (4,0) -- (5.5,-1.5) (4,1) -- (5.5,2.5);
                \draw [ultra thick,->,red,shorten <= 2pt,shorten >= 2pt] (1.5,0) -- (1.5,1);
                
                \draw node at (2,-.75) {$v$};
                \draw node at (2,1.75) {$w$};
                \draw node at (2,-2.75) {$\Sigma$};
                
                \draw [draw=none,pattern={Lines[angle=-45,distance=3pt]},pattern color=NavyBlue] (12,-.5) -- (12,1.5) -- (10.5,1) -- (10.5,0) -- (12,-.5) -- (12,1.5);
                \draw [draw=none,pattern={Lines[angle=-45,distance=3pt]},pattern color=ForestGreen] (12,-.5) -- (12,1.5) -- (13.5,1) -- (13.5,0) -- (12,-.5) -- (12,1.5);
                
                \draw node [draw,circle,fill,inner sep=1.5pt,label=right:$v$] (v) at (12,-.5) {};
                \draw node [draw,circle,fill,inner sep=1.5pt,label=right:$w$] (w) at (12,1.5) {};
                \draw [ultra thick,red] (v) -- (w) coordinate [midway] (mid);
                \draw [ultra thick,red,->] (v) -- (mid);
                
                \draw node at (12,-1.75) {$X$};
                
                \draw [->] (8,.5) -- (9.6,.5) node [midway,above] {$f$};
                
                \draw [<-,dashed] (1,1.75) to [bend left=30] (3,3.5) node [right] (lab-vd) {\emph{vertex discs}};
                \draw [<-,dashed] (3,-.75) to [bend right=10] (lab-vd);
                \draw [<-,dashed] (-1,.5) to [bend right=30] (3.5,-2) node [right] (lab-cd) {\emph{cellular discs}};
                \draw [<-,dashed] (5,.5) to [bend left=20] (lab-cd.15);
                \draw [<-,dashed] (1,.75) to [bend right=20] (-2.5,2) node [left] {\emph{a $1$-handle}};
            \end{tikzpicture}
        \caption{Pieces of a transverse incompressible admissible surface.}
        \label{fig:transv-surf}
    \end{figure}

    \subsection{Connectedness of links}\label{subsec:conn-lks}
    
    We henceforth assume that the $2$-complex $X$ is a cellulated, oriented, compact, connected surface, and we'll denote it by $S$.
    
    Consider an admissible surface $f:(\Sigma,\partial\Sigma)\rightarrow(S,c)$ with $f_*[\Sigma]=n(\Sigma)\alpha$ in $H_2(S,c)$.
    We may assume that $f$ is transverse, incompressible, monotone, and disc- and sphere-free, as explained in \S{}\ref{subsec:incomp-monot} and \S{}\ref{subsec:transv}.
    
    We now want to use the fact that $S$ is a surface to ensure that $\Sigma$ is `thick enough', in the sense that its vertex discs have connected links.
    
    More precisely, consider the $2$-complex $\bar{\Sigma}$ obtained from $\Sigma$ by collapsing all vertex discs to vertices and all $1$-handles to edges --- hence, $f$ induces a combinatorial map $\bar{f}:\bar{\Sigma}\rightarrow S$, but $\bar{\Sigma}$ may not be a surface (see for instance Figure \ref{fig:disc-lnk}: the vertex disc at the centre of $\Sigma$ becomes a disconnecting vertex in $\bar{\Sigma}$).
    \begin{defi}
        We say that a transverse incompressible admissible surface $f:(\Sigma,\partial\Sigma)\rightarrow(S,c)$ has \emph{connected links} if the $2$-complex $\bar{\Sigma}$ has connected vertex links.
    \end{defi}
    
    If $f$ does not have connected links, let $D$ be a vertex disc in $\Sigma$ whose corresponding vertex in $\bar{\Sigma}$ has disconnected link.
    Let $v$ be the image of $D$ under $f$; so $v$ is a vertex of $S$.
    There are two cases: either $v$ lies in the interior of $S$, or on the boundary.
    
    Assume first that $v$ lies in the interior of $S$; an example is depicted in Figure \ref{fig:disc-lnk} (all the $2$-cells on the picture are triangles for simplicity, but the general case is similar).
    \begin{figure}[htb]
        \centering
            \centering
            \begin{tikzpicture}[every node/.style={rectangle,draw=none,fill=none},scale=.5]
                    (-1.25,-1.25) -- (.25,-2.75) (5.25,-1.25) -- (3.75,-2.75);
                \draw [ultra thick,->,red,shorten <= 1pt,shorten >= 1pt] (-.25,1.75) -- (-.75,2.25);
                \draw [ultra thick,->,red,shorten <= 1pt,shorten >= 1pt] (4.25,1.75) -- (4.75,2.25);
                \draw [ultra thick,->,red,shorten <= 1pt,shorten >= 1pt] (-.25,-1.75) -- (-.75,-2.25);
                \draw [ultra thick,->,red,shorten <= 1pt,shorten >= 1pt] (4.25,-1.75) -- (4.75,-2.25);
                \draw [draw=none,pattern={Lines[angle=-60,distance=4pt]},pattern color=NavyBlue] (-3,1.5) -- (-1.5,1.5) -- (-1,1) -- (-1,-1) -- (-1.5,-1.5) -- (-3,-1.5) -- cycle;
                \draw [draw=none,pattern={Lines[angle=-60,distance=4pt]},pattern color=ForestGreen] (7,1.5) -- (5.5,1.5) -- (5,1) -- (5,-1) -- (5.5,-1.5) -- (7,-1.5) -- cycle;
                \draw [draw=none,fill=gray!30] (-1.5,4.5) -- (0,3) -- (-1.5,1.5) -- (-3,1.5)
                    (-1.5,-4.5) -- (0,-3) -- (-1.5,-1.5) -- (-3,-1.5)
                    (5.5,4.5) -- (4,3) -- (5.5,1.5) -- (7,1.5)
                    (5.5,-4.5) -- (4,-3) -- (5.5,-1.5) -- (7,-1.5);
                \draw [thick,fill=gray!30] (-1,1) -- (-1,-1) -- (.5,-2.5) -- (3.5,-2.5) -- (5,-1) -- (5,1) -- (3.5,2.5) -- (.5,2.5) -- cycle;
                \draw [thick] (0,3) -- (-1.5,1.5) -- (-3,1.5) (4,3) -- (5.5,1.5) -- (7,1.5)
                    (0,-3) -- (-1.5,-1.5) -- (-3,-1.5) (4,-3) -- (5.5,-1.5) -- (7,-1.5);
                \draw [ultra thick] (-1.5,4.5) -- (.5,2.5) -- (3.5,2.5) -- (5.5,4.5)
                    (-1.5,-4.5) -- (.5,-2.5) -- (3.5,-2.5) -- (5.5,-4.5);
                    
                \draw node at (0,4) {$\partial\Sigma$};
                \draw node at (2,-4) {$\Sigma$};
                \draw node at (2,0) {$v$};
                \draw node at (-1.5,2.75) {$u_1$};
                \draw node at (5.5,2.75) {$u_3$};
                \draw node at (-1.5,-2.75) {$u_2$};
                \draw node at (5.5,-2.75) {$u_4$};
                
                \draw [<-,dashed] (1.75,1.25) to [bend left=15] (3,4) node [above right] {$D$};
                
                \draw [draw=none,pattern={Lines[angle=-60,distance=4pt]},pattern color=BurntOrange] (11,1.5) -- (12.5,0) -- (14,1.5);
                \draw [draw=none,pattern={Lines[angle=-60,distance=4pt]},pattern color=NavyBlue] (11,1.5) -- (12.5,0) -- (11,-1.5);
                \draw [draw=none,pattern={Lines[angle=-60,distance=4pt]},pattern color=Fuchsia] (11,-1.5) -- (12.5,0) -- (14,-1.5);
                \draw [draw=none,pattern={Lines[angle=-60,distance=4pt]},pattern color=ForestGreen] (14,1.5) -- (12.5,0) -- (14,-1.5);
                \draw node [draw,circle,fill,inner sep=1.5pt] (v) at (12.5,0) {};
                \draw node [draw,circle,fill,inner sep=1.5pt,label=above right:$u_4$] (u4) at (14,1.5) {};
                \draw node [draw,circle,fill,inner sep=1.5pt,label=above left:$u_1$] (u1) at (11,1.5) {};
                \draw node [draw,circle,fill,inner sep=1.5pt,label=below right:$u_3$] (u3) at (14,-1.5) {};
                \draw node [draw,circle,fill,inner sep=1.5pt,label=below left:$u_2$] (u2) at (11,-1.5) {};
                \foreach \i in {1,...,4}{\draw [ultra thick,red] (v) -- (u\i) coordinate [midway] (mid\i);
                    \draw[ultra thick,red,->] (v) -- (mid\i);}
                \draw [ultra thick,red] (u4) -- (u1) coordinate [midway] (mid);
                \draw[ultra thick,red,->] (u4) -- (mid);
                \draw [ultra thick,red] (u3) -- (u4) coordinate [midway] (mid);
                \draw[ultra thick,red,->] (u3) -- (mid);
                \draw [ultra thick,red] (u2) -- (u3) coordinate [midway] (mid);
                \draw[ultra thick,red,->] (u2) -- (mid);
                \draw [ultra thick,red] (u1) -- (u2) coordinate [midway] (mid);
                \draw[ultra thick,red,->] (u1) -- (mid);
                
                \draw node [circle,fill=white,inner sep=.1pt] at (12.5,1) {{\small$\sigma_1$}};
                \draw node [circle,fill=white,inner sep=.1pt] at (13.5,0) {{\small$\sigma_2$}};
                \draw node [circle,fill=white,inner sep=.1pt] at (12,0) {$v$};
                \draw node at (12.5,-3) {$S$};
                
                \draw [->] (8,0) -- (9.6,0) node [midway,above] {$f$};
            \end{tikzpicture}
        \caption{A vertex disc with disconnected link mapping to a vertex in the interior of $S$ (note that the map $f$ is orientation-preserving on the blue cellular disc but orientation-reversing on the green one).}
        \label{fig:disc-lnk}
    \end{figure}
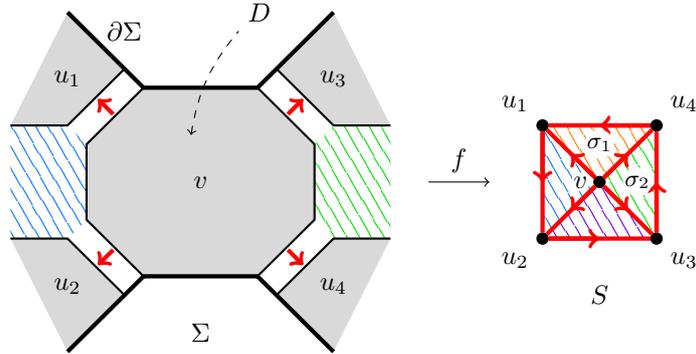
    The main point is that there is a $2$-cell in $S$ between any two consecutive edges around $v$; in particular, there is a sequence of $2$-cells $\sigma_1,\dots,\sigma_\ell$ lying between two successive edges whose preimages are $1$-handles with one end each on $\partial\Sigma$, as in Figure \ref{fig:disc-lnk}.
    Now we perform a homotopy that moves the image of $\partial\Sigma$ across the $2$-cells $\sigma_1,\dots,\sigma_\ell$: see Figure \ref{fig:lnk-conn}.
    \begin{figure}[htb]
        \centering
            \centering
            \begin{tikzpicture}[every node/.style={rectangle,draw=none,fill=none},scale=.45]
                \draw [draw=none,pattern={Lines[angle=-60,distance=4pt]},pattern color=NavyBlue] (-3,1.5) -- (-1.5,1.5) -- (-1,1) -- (-1,-1) -- (-1.5,-1.5) -- (-3,-1.5) -- cycle;
                \draw [draw=none,pattern={Lines[angle=-60,distance=4pt]},pattern color=ForestGreen] (7,1.5) -- (5.5,1.5) -- (5,1) -- (5,-1) -- (5.5,-1.5) -- (7,-1.5) -- cycle;
                    (-1.25,-1.25) -- (.25,-2.75) (5.25,-1.25) -- (3.75,-2.75);
                \draw [ultra thick,->,red,shorten <= 1pt,shorten >= 1pt] (-.25,1.75) -- (-.75,2.25);
                \draw [ultra thick,->,red,shorten <= 1pt,shorten >= 1pt] (4.25,1.75) -- (4.75,2.25);
                \draw [ultra thick,->,red,shorten <= 1pt,shorten >= 1pt] (-.25,-1.75) -- (-.75,-2.25);
                \draw [ultra thick,->,red,shorten <= 1pt,shorten >= 1pt] (4.25,-1.75) -- (4.75,-2.25);
                \draw [draw=none,fill=gray!30] (-1.5,4.5) -- (0,3) -- (-1.5,1.5) -- (-3,1.5)
                    (-1.5,-4.5) -- (0,-3) -- (-1.5,-1.5) -- (-3,-1.5)
                    (5.5,4.5) -- (4,3) -- (5.5,1.5) -- (7,1.5)
                    (5.5,-4.5) -- (4,-3) -- (5.5,-1.5) -- (7,-1.5);
                \draw [thick,fill=gray!30] (-1,1) -- (-1,-1) -- (.5,-2.5) -- (3.5,-2.5) -- (5,-1) -- (5,1) -- (3.5,2.5) -- (.5,2.5) -- cycle;
                \draw [thick] (0,3) -- (-1.5,1.5) -- (-3,1.5) (4,3) -- (5.5,1.5) -- (7,1.5)
                    (0,-3) -- (-1.5,-1.5) -- (-3,-1.5) (4,-3) -- (5.5,-1.5) -- (7,-1.5);
                \draw [ultra thick] (-1.5,4.5) -- (.5,2.5) -- (3.5,2.5) -- (5.5,4.5)
                    (-1.5,-4.5) -- (.5,-2.5) -- (3.5,-2.5) -- (5.5,-4.5);
                    
                \draw node at (2,0) {$v$};
                \draw node at (-1.5,2.75) {$u_1$};
                \draw node at (5.5,2.75) {$u_3$};
                \draw node at (-1.5,-2.75) {$u_2$};
                \draw node at (5.5,-2.75) {$u_4$};
                
                \draw node at (9,0) {{\Huge$\rightsquigarrow$}};
                
                \draw [draw=none,pattern={Lines[angle=-60,distance=4pt]},pattern color=NavyBlue] (11,1.5) -- (12.5,1.5) -- (13,.5) -- (13,-1) -- (12.5,-1.5) -- (11,-1.5) -- cycle;
                \draw [draw=none,pattern={Lines[angle=-60,distance=4pt]},pattern color=ForestGreen] (21,1.5) -- (19.5,1.5) -- (19,.5) -- (19,-1) -- (19.5,-1.5) -- (21,-1.5) -- cycle;
                \draw [draw=none,pattern={Lines[angle=-60,distance=4pt]},pattern color=BurntOrange] (13.5,1.5) -- (15.5,1.5) -- (15.5,.5) -- (14,.5) -- cycle;
                \draw [draw=none,pattern={Lines[angle=-60,distance=4pt]},pattern color=ForestGreen] (18.5,1.5) -- (16.5,1.5) -- (16.5,.5) -- (18,.5) -- cycle;
                \draw [ultra thick,->,red,shorten <= 1pt,shorten >= 1pt] (18.5,.5) -- (19,1.5);
                \draw [ultra thick,->,red,shorten <= 1pt,shorten >= 1pt] (13.5,.5) -- (13,1.5);
                \draw [ultra thick,->,red,shorten <= 1pt,shorten >= 1pt] (16,.5) -- (16,1.5);
                \draw [ultra thick,<-,red,shorten <= 1pt,shorten >= 1pt] (14.5,2) -- (15.5,2);
                \draw [ultra thick,->,red,shorten <= 1pt,shorten >= 1pt] (17.5,2) -- (16.5,2);
                \draw [ultra thick,->,red,shorten <= 1pt,shorten >= 1pt] (13.75,-1.75) -- (13.25,-2.25);
                \draw [ultra thick,->,red,shorten <= 1pt,shorten >= 1pt] (18.25,-1.75) -- (18.75,-2.25);
                \draw [draw=none,fill=gray!30] (12.5,4.5) -- (14.5,2.5) -- (14.5,1.5) -- (11,1.5)
                    (12.5,-4.5) -- (14,-3) -- (12.5,-1.5) -- (11,-1.5)
                    (19.5,4.5) -- (17.5,2.5) -- (17.5,1.5) -- (21,1.5)
                    (19.5,-4.5) -- (18,-3) -- (19.5,-1.5) -- (21,-1.5);
                \draw [thick,fill=gray!30] (13,-1) -- (14.5,-2.5) -- (17.5,-2.5) -- (19,-1) -- (19,.5) -- (17.5,.5) -- (14.5,.5) -- (13,.5) -- cycle
                    (16.5,2.5) -- (16.5,1.5) -- (15.5,1.5) -- (15.5,2.5) -- cycle;
                \draw [thick] (14.5,2.5) -- (14.5,1.5) -- (12.5,1.5) -- (11,1.5)
                    (17.5,2.5) -- (17.5,1.5) -- (19.5,1.5) -- (21,1.5)
                    (14,-3) -- (12.5,-1.5) -- (11,-1.5) (18,-3) -- (19.5,-1.5) -- (21,-1.5);
                \draw [ultra thick] (12.5,4.5) -- (14.5,2.5) -- (17.5,2.5) -- (19.5,4.5)
                    (12.5,-4.5) -- (14.5,-2.5) -- (17.5,-2.5) -- (19.5,-4.5);
                    
                \draw node at (16,-1) {$v$};
                \draw node at (12.5,2.75) {$u_1$};
                \draw node at (19.5,2.75) {$u_3$};
                \draw node at (12.5,-2.75) {$u_2$};
                \draw node at (19.5,-2.75) {$u_4$};
                \draw node at (16,2) {$u_4$};
                \draw node [circle,fill=white,inner sep=.1pt] at (14.5,1) {{\small$\sigma_1$}};
                \draw node [circle,fill=white,inner sep=.1pt] at (17.5,1) {{\small$\sigma_2$}};
            \end{tikzpicture}
        \caption{Making links of vertex discs connected (interior case).}
        \label{fig:lnk-conn}
    \end{figure}
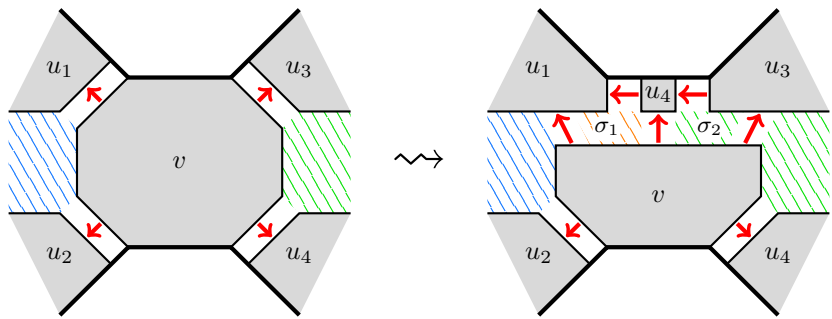
    The new map $f:\Sigma\rightarrow S$ defines an admissible surface $\left(\Sigma,\partial\Sigma\right)\rightarrow\left(S,c\right)$.
    Note that $f$ has been modified by a homotopy, so we still have $f_*[\Sigma]=n(\Sigma)\alpha$.
    
    This operation decreases the number of connected components in the link of $D$ (or more precisely, of its image in $\bar{\Sigma}$).
    Note that new vertex discs may have been created (such as the one mapping to $u_4$ in Figure \ref{fig:lnk-conn}), but they all have connected link.
    As for preexisting vertex discs of $\Sigma$, the operation doesn't impact the number of connected components of their links.
    Hence, if $\left\{D_i\right\}_i$ is the set of vertex discs of $\Sigma$ mapping to the interior of $S$, and $k_i$ denotes the number of connected components of the link of $D_i$, then we have made the quantity $\sum_i\left(k_i-1\right)$ decrease strictly.
    We may therefore iterate to ensure that all vertex discs mapping to the interior of $S$ have connected link.
    
    We deal with the case where $v$ lies on the boundary in the following way.
    We thicken $S$ by gluing a cellulated annulus to each of its boundary components.
    This modifies the cellular structure of $S$ but not its homeomorphism type (in other words, $S$ has been replaced with another cell complex $S'$ with more cells, but with $S'$ homeomorphic to $S$), and this preserves all the properties of the map $f$ --- in particular, $f$ is transverse for the new cellulation of $S$.
    Now all of the vertex discs with disconnected link map to the interior of $S$.
    Therefore, we can apply the operation described above to make all their links connected.
    This may create some new vertex discs mapping to the boundary, but they will have connected link.
    Hence, after both operations, all vertex discs have connected link.
    
    We therefore obtain the following:
    
    \begin{lmm}[Admissible surfaces with connected links]\label{lmm:conn-lks}        
        Fix $\alpha\in H_2\left(S,c\right)$, and let $f:(\Sigma,\partial\Sigma)\rightarrow(S,c)$ be a transverse incompressible admissible surface with $f_*[\Sigma]=n(\Sigma)\alpha$.
        Then, after possibly changing the cellular structure on $S$, the map $f$ may be homotoped to a transverse incompressible admissible surface with connected links.\qed
    \end{lmm}

    Note that since our admissible surface was modified by a homotopy, the properties of being incompressible, monotone, and disc- and sphere-free are preserved.

    \subsection{Folding}\label{subsec:or-nmix}
    
    The key properties of admissible surfaces that we will need in the surface group case are related to orientation.
    Indeed, both $S$ and the admissible surface $\Sigma$ are oriented.
    Since $f$ is transverse and cellular discs map homeomorphically into $S$, they can be of two types: either they preserve the orientation or they reverse it.
    Having cellular discs of opposite orientations is undesirable, and we are now going to modify $\Sigma$ to avoid this situation as much as possible.
    
    We assume that $f:\left(\Sigma,\partial\Sigma\right)\rightarrow\left(S,c\right)$ is transverse, incompressible, monotone, disc- and sphere-free, and with connected links.
    
    Suppose first that there is a connected component of $\Sigma$ containing cellular discs of two different types --- i.e.\ one is orientation-preserving with respect to $f$ and the other is orientation-reversing.
    Since $f:\left(\Sigma,\partial\Sigma\right)\rightarrow\left(S,c\right)$ has connected links, any two cellular discs in $\Sigma$ mapping to cells in $S$ with a common vertex on their boundaries must be connected by a path of cellular discs and $1$-handles in $\Sigma$.
    It follows that any two cellular discs in the same connected component of $\Sigma$ are connected by a path of cellular discs and $1$-handles.
    
    Therefore, $\Sigma$ must contain two cellular discs of opposite types that are adjacent via a $1$-handle.
    Since $S$ is a surface, those two cellular discs must map to the same $2$-cell of $S$, and we are in the situation of Figure \ref{fig:cells-opp-or} (pictures are given for the case of a $2$-cell of degree $3$, but the general case is similar) --- in other words, $f$ \emph{folds} those two cellular discs onto one another.    
    \begin{figure}[htb]
        \centering
            \centering
            \begin{tikzpicture}[every node/.style={rectangle,draw=none,fill=none},scale=.5]
                \draw [draw=none,pattern={Lines[angle=-60,distance=4pt]},pattern color=NavyBlue]
                    (0,.5) -- (-1.5,1.5) -- (-2,1) -- (-2,-1) -- (-1.5,-1.5) -- (0,-.5) -- cycle;
                \draw [draw=none,pattern={Lines[angle=-60,distance=4pt]},pattern color=NavyBlue]
                    (4,.5) -- (5.5,1.5) -- (6,1) -- (6,-1) -- (5.5,-1.5) -- (4,-.5) -- cycle;
                \draw [draw=none,fill=gray!30] (0,-.5) -- (4,-.5) -- (5.5,-1.5) -- (7.5,-3.5) -- (-3.5,-3.5) -- (-1.5,-1.5) -- cycle;
                \draw [draw=none,fill=gray!30] (0,.5) -- (4,.5) -- (5.5,1.5) -- (7.5,3.5) -- (-3.5,3.5) -- (-1.5,1.5) -- cycle;
                \draw [draw=none,fill=gray!30] (-4,3) -- (-2,1) -- (-2,-1) -- (-4,-3) -- cycle;
                \draw [draw=none,fill=gray!30] (8,3) -- (6,1) -- (6,-1) -- (8,-3) -- cycle;
                \draw [thick] (0,.5) -- (4,.5) (0,-.5) -- (4,-.5)
                    (0,-.5) -- (-1.5,-1.5) -- (-3.5,-3.5)
                    (0,.5) -- (-1.5,1.5) -- (-3.5,3.5)
                    (-4,3) -- (-2,1) -- (-2,-1) -- (-4,-3)
                    (4,-.5) -- (5.5,-1.5) -- (7.5,-3.5)
                    (4,.5) -- (5.5,1.5) -- (7.5,3.5)
                    (8,3) -- (6,1) -- (6,-1) -- (8,-3);
                \draw [ultra thick,->,red,shorten <= 1pt,shorten >= 1pt] (1.5,-.5) -- (1.5,.5) node [above] {{\small$e$}};
                \draw [ultra thick,->,red,shorten <= 1pt,shorten >= 1pt] (-2.5,1.5) -- (-2,2) node [above] {{\small$g$}};
                \draw [ultra thick,->,red,shorten <= 1pt,shorten >= 1pt] (-2,-2) -- (-2.5,-1.5) node [above] {{\small$f$}};
                \draw [ultra thick,->,red,shorten <= 1pt,shorten >= 1pt] (6.5,1.5) -- (6,2) node [above] {{\small$g$}};
                \draw [ultra thick,->,red,shorten <= 1pt,shorten >= 1pt] (6,-2) -- (6.5,-1.5) node [above] {{\small$f$}};
                
                \draw node at (-1,0) {\LARGE$\circlearrowleft$};
                \draw node at (5,0) {\LARGE$\circlearrowright$};
                \draw node at (2,-2) {$u$};
                \draw node at (2,2) {$w$};
                \draw node at (-3,0) {$v$};
                \draw node at (7,0) {$v$};
                \draw node at (2,-4) {$\Sigma$};
                
                \draw [draw=none,pattern={Lines[angle=-60,distance=4pt]},pattern color=NavyBlue] (14.5,-1.5) -- (12,0) -- (14.5,1.5) -- cycle;
                
                \draw node [draw,circle,fill,inner sep=1.5pt,label=below right:$u$] (u) at (14.5,-1.5) {};
                \draw node [draw,circle,fill,inner sep=1.5pt,label=left:$v$] (v) at (12,0) {};
                \draw node [draw,circle,fill,inner sep=1.5pt,label=above right:$w$] (w) at (14.5,1.5) {};
                \draw [ultra thick,red] (u) -- (v) coordinate [midway] (mid1);
                \draw [ultra thick,red,->] (u) -- (mid1) node [below left] {{\small$f$}};
                \draw [ultra thick,red] (v) -- (w) coordinate [midway] (mid2);
                \draw [ultra thick,red,->] (v) -- (mid2) node [above left] {{\small$g$}};
                \draw [ultra thick,red] (u) -- (w) coordinate [midway] (mid3);
                \draw [ultra thick,red,->] (u) -- (mid3) node [right] {{\small$e$}};
                
                \draw node at (13.66,0) {\LARGE$\circlearrowleft$};
                
                \draw node at (13,-3) {$S$};
                
                \draw [->] (9,0) -- (10.6,0) node [midway,above] {$f$};
            \end{tikzpicture}
        \caption{Adjacent cellular discs of opposite orientations.}
        \label{fig:cells-opp-or}
    \end{figure}
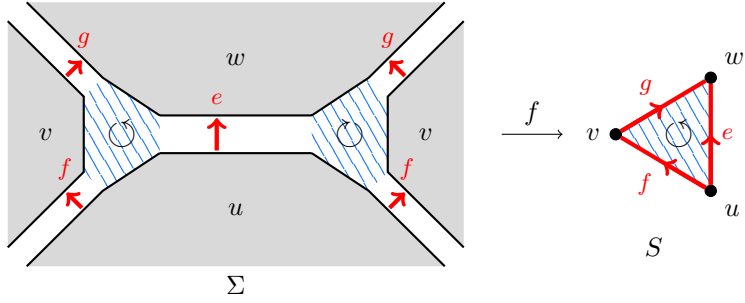
    
    In this case, we can delete the two adjacent cellular discs as illustrated in Figure \ref{fig:elim-opp-or}.
    \begin{figure}[htb]
        \centering
        \begin{tikzpicture}[every node/.style={rectangle,draw=none,fill=none},scale=.43]
            \draw [draw=none,pattern={Lines[angle=-60,distance=4pt]},pattern color=NavyBlue]
                (0,.5) -- (-1.5,1.5) -- (-2,1) -- (-2,-1) -- (-1.5,-1.5) -- (0,-.5) -- cycle;
            \draw [draw=none,pattern={Lines[angle=-60,distance=4pt]},pattern color=NavyBlue]
                (4,.5) -- (5.5,1.5) -- (6,1) -- (6,-1) -- (5.5,-1.5) -- (4,-.5) -- cycle;
            \draw [draw=none,fill=gray!30] (0,-.5) -- (4,-.5) -- (5.5,-1.5) -- (7.5,-3.5) -- (-3.5,-3.5) -- (-1.5,-1.5) -- cycle;
            \draw [draw=none,fill=gray!30] (0,.5) -- (4,.5) -- (5.5,1.5) -- (7.5,3.5) -- (-3.5,3.5) -- (-1.5,1.5) -- cycle;
            \draw [draw=none,fill=gray!30] (-4,3) -- (-2,1) -- (-2,-1) -- (-4,-3) -- cycle;
            \draw [draw=none,fill=gray!30] (8,3) -- (6,1) -- (6,-1) -- (8,-3) -- cycle;
            \draw [thick] (0,.5) -- (4,.5) (0,-.5) -- (4,-.5)
                (0,-.5) -- (-1.5,-1.5) -- (-3.5,-3.5)
                (0,.5) -- (-1.5,1.5) -- (-3.5,3.5)
                (-4,3) -- (-2,1) -- (-2,-1) -- (-4,-3)
                (4,-.5) -- (5.5,-1.5) -- (7.5,-3.5)
                (4,.5) -- (5.5,1.5) -- (7.5,3.5)
                (8,3) -- (6,1) -- (6,-1) -- (8,-3);
            \draw [ultra thick,->,red,shorten <= 1pt,shorten >= 1pt] (1.5,-.5) -- (1.5,.5) node [above] {{\small$e$}};
            \draw [ultra thick,->,red,shorten <= 1pt,shorten >= 1pt] (-2.5,1.5) -- (-2,2) node [above] {{\small$g$}};
            \draw [ultra thick,->,red,shorten <= 1pt,shorten >= 1pt] (-2,-2) -- (-2.5,-1.5) node [above] {{\small$f$}};
            \draw [ultra thick,->,red,shorten <= 1pt,shorten >= 1pt] (6.5,1.5) -- (6,2) node [above] {{\small$g$}};
            \draw [ultra thick,->,red,shorten <= 1pt,shorten >= 1pt] (6,-2) -- (6.5,-1.5) node [above] {{\small$f$}};
            
            \draw node at (-1,0) {\LARGE$\circlearrowleft$};
            \draw node at (5,0) {\LARGE$\circlearrowright$};
            \draw node at (2,-2) {$u$};
            \draw node at (2,2) {$w$};
            \draw node at (-3,0) {$v$};
            \draw node at (7,0) {$v$};
            
            \draw node at (10,0) {{\Huge$\rightsquigarrow$}};
            
            \draw [draw=none,fill=gray!30] (12.5,3.5) -- (15.5,1.5) -- (20.5,1.5) -- (23.5,3.5) -- cycle;
            \draw [draw=none,fill=gray!30] (12.5,-3.5) -- (15.5,-1.5) -- (20.5,-1.5) -- (23.5,-3.5) -- cycle;
            \draw [draw=none,fill=gray!30] (12,3) -- (15,1) -- (21,1) -- (24,3) -- (24,-3) -- (21,-1) -- (15,-1) -- (12,-3) -- cycle;
            \draw [thick] (12.5,3.5) -- (15.5,1.5) -- (20.5,1.5) -- (23.5,3.5) (12,3) -- (15,1) -- (21,1) -- (24,3)
                (12.5,-3.5) -- (15.5,-1.5) -- (20.5,-1.5) -- (23.5,-3.5) (12,-3) -- (15,-1) -- (21,-1) -- (24,-3);
            \draw [ultra thick,->,red,shorten <= 1pt,shorten >= 1pt] (19,1) -- (19,1.5) node [above] {{\small$g$}};
            \draw [ultra thick,->,red,shorten <= 1pt,shorten >= 1pt] (19,-1.5) -- (19,-1) node [above] {{\small$f$}};
            
            \draw node at (18,2.75) {$w$};
            \draw node at (18,0) {$v$};
            \draw node at (18,-2.75) {$u$};
        \end{tikzpicture}
        \caption{Eliminating cellular discs of opposite orientations.}
        \label{fig:elim-opp-or}
    \end{figure}
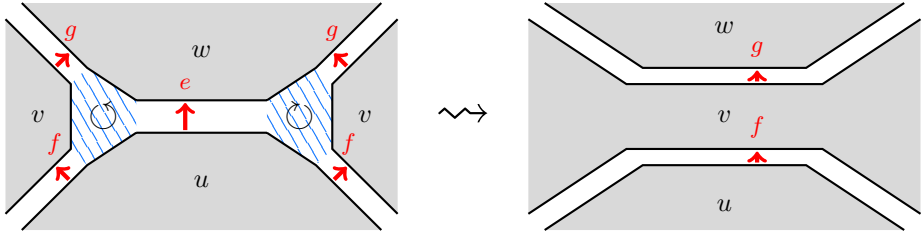
    This operation does not change the homotopy type of $\Sigma$ nor the boundary map $\partial f:\partial\Sigma\rightarrow\coprod_iS^1$.
    It also preserves the class $f_*[\Sigma]=n(\Sigma)\alpha$, as well as transversality, incompressibility, monotonicity of $f$, and the property of being disc- and sphere-free.
    Moreover, it makes the number of cellular discs of $\Sigma$ decrease strictly.
    Hence, after repeating a finite number of times, we may assume that each connected component of $\Sigma$ contains cellular discs of only one type: either orientation-preserving or orientation-reversing.
    
    In other words, we have reduced to the case where $\Sigma$ has the following property:
    \begin{defi}
        Let $S$ be a cellulated, oriented, compact, connected surface. A transverse incompressible admissible surface $f:(\Sigma,\partial\Sigma)\rightarrow(S,c)$ is said to be \emph{non-folded} if each connected component of $\Sigma$ contains cellular discs that are either all orientation-preserving or all orientation-reversing.
    \end{defi}
    
    However, the operation just described could have created some vertex discs with disconnected link in $\Sigma$ (as removing cellular discs amounts to deleting edges in the links of vertex discs).
    
    To fix this, we would like to successively apply the operation described above and the one of \S{}\ref{subsec:conn-lks}.
    This relies on the following crucial observation: in the process described in \S{}\ref{subsec:conn-lks}, we can always choose the orientation of the cellular discs that we add.
    Indeed, the added cellular discs correspond to a path in the link of $v$ in $S$, and this link is a circle, so there are two possible paths, one corresponding to adding cellular discs of positive orientation to $\Sigma$, and the other corresponding to adding cellular discs of negative orientation.
    For example, Figure \ref{fig:2-chcs} shows two possible choices that make the link of the vertex disc of Figure \ref{fig:disc-lnk} connected.
    In particular, when applying the operation of \S{}\ref{subsec:conn-lks}, we can assume that we are only adding cellular discs of \emph{positive} orientation.
    \begin{figure}[htb]
        \centering
            \centering
            \begin{tikzpicture}[every node/.style={rectangle,draw=none,fill=none},scale=.45]
                \draw [draw=none,pattern={Lines[angle=-60,distance=4pt]},pattern color=NavyBlue] (-3,1.5) -- (-1.5,1.5) -- (-1,1) -- (-1,-1) -- (-1.5,-1.5) -- (-3,-1.5) -- cycle;
                \draw [draw=none,pattern={Lines[angle=-60,distance=4pt]},pattern color=ForestGreen] (7,1.5) -- (5.5,1.5) -- (5,1) -- (5,-1) -- (5.5,-1.5) -- (7,-1.5) -- cycle;
                \draw [draw=none,pattern={Lines[angle=-60,distance=4pt]},pattern color=NavyBlue] (-.5,1.5) -- (1.5,1.5) -- (1.5,.5) -- (0,.5) -- cycle;
                \draw [draw=none,pattern={Lines[angle=-60,distance=4pt]},pattern color=Fuchsia] (4.5,1.5) -- (2.5,1.5) -- (2.5,.5) -- (4,.5) -- cycle;
                \draw [ultra thick,->,red,shorten <= 1pt,shorten >= 1pt] (4.5,.5) -- (5,1.5);
                \draw [ultra thick,->,red,shorten <= 1pt,shorten >= 1pt] (-.5,.5) -- (-1,1.5);
                \draw [ultra thick,->,red,shorten <= 1pt,shorten >= 1pt] (2,.5) -- (2,1.5);
                \draw [ultra thick,->,red,shorten <= 1pt,shorten >= 1pt] (.5,2) -- (1.5,2);
                \draw [ultra thick,<-,red,shorten <= 1pt,shorten >= 1pt] (3.5,2) -- (2.5,2);
                \draw [ultra thick,->,red,shorten <= 1pt,shorten >= 1pt] (-.25,-1.75) -- (-.75,-2.25);
                \draw [ultra thick,->,red,shorten <= 1pt,shorten >= 1pt] (4.25,-1.75) -- (4.75,-2.25);
                \draw [draw=none,fill=gray!30] (-1.5,4.5) -- (.5,2.5) -- (.5,1.5) -- (-3,1.5)
                    (-1.5,-4.5) -- (0,-3) -- (-1.5,-1.5) -- (-3,-1.5)
                    (5.5,4.5) -- (3.5,2.5) -- (3.5,1.5) -- (7,1.5)
                    (5.5,-4.5) -- (4,-3) -- (5.5,-1.5) -- (7,-1.5);
                \draw [thick,fill=gray!30] (-1,.5) -- (-1,-1) -- (.5,-2.5) -- (3.5,-2.5) -- (5,-1) -- (5,.5) -- (3.5,.5) -- (.5,.5) -- cycle;
                \draw [thick,fill=gray!30] (1.5,2.5) -- (1.5,1.5) -- (2.5,1.5) -- (2.5,2.5) -- cycle;
                \draw [thick] (.5,2.5) -- (.5,1.5) -- (-3,1.5) (3.5,2.5) -- (3.5,1.5) -- (7,1.5)
                    (0,-3) -- (-1.5,-1.5) -- (-3,-1.5) (4,-3) -- (5.5,-1.5) -- (7,-1.5);
                \draw [ultra thick] (-1.5,4.5) -- (.5,2.5) -- (3.5,2.5) -- (5.5,4.5)
                    (-1.5,-4.5) -- (.5,-2.5) -- (3.5,-2.5) -- (5.5,-4.5);
                    
                \draw node at (2,-1) {$v$};
                \draw node at (-1.5,2.75) {$u_1$};
                \draw node at (5.5,2.75) {$u_3$};
                \draw node at (-1.5,-2.75) {$u_2$};
                \draw node at (5.5,-2.75) {$u_4$};
                \draw node at (2,2) {$u_2$};
                \draw node at (-2,0) {\LARGE$\circlearrowleft$};
                \draw node at (6,0) {\LARGE$\circlearrowright$};
                \draw node at (.5,1) {\Large$\circlearrowright$};
                \draw node at (3.5,1) {\Large$\circlearrowright$};
                
                \draw node at (9,0) {{\large or}};
                
                \draw [draw=none,pattern={Lines[angle=-60,distance=4pt]},pattern color=NavyBlue] (11,1.5) -- (12.5,1.5) -- (13,.5) -- (13,-1) -- (12.5,-1.5) -- (11,-1.5) -- cycle;
                \draw [draw=none,pattern={Lines[angle=-60,distance=4pt]},pattern color=ForestGreen] (21,1.5) -- (19.5,1.5) -- (19,.5) -- (19,-1) -- (19.5,-1.5) -- (21,-1.5) -- cycle;
                \draw [draw=none,pattern={Lines[angle=-60,distance=4pt]},pattern color=BurntOrange] (13.5,1.5) -- (15.5,1.5) -- (15.5,.5) -- (14,.5) -- cycle;
                \draw [draw=none,pattern={Lines[angle=-60,distance=4pt]},pattern color=ForestGreen] (18.5,1.5) -- (16.5,1.5) -- (16.5,.5) -- (18,.5) -- cycle;
                \draw [ultra thick,->,red,shorten <= 1pt,shorten >= 1pt] (18.5,.5) -- (19,1.5);
                \draw [ultra thick,->,red,shorten <= 1pt,shorten >= 1pt] (13.5,.5) -- (13,1.5);
                \draw [ultra thick,->,red,shorten <= 1pt,shorten >= 1pt] (16,.5) -- (16,1.5);
                \draw [ultra thick,<-,red,shorten <= 1pt,shorten >= 1pt] (14.5,2) -- (15.5,2);
                \draw [ultra thick,->,red,shorten <= 1pt,shorten >= 1pt] (17.5,2) -- (16.5,2);
                \draw [ultra thick,->,red,shorten <= 1pt,shorten >= 1pt] (13.75,-1.75) -- (13.25,-2.25);
                \draw [ultra thick,->,red,shorten <= 1pt,shorten >= 1pt] (18.25,-1.75) -- (18.75,-2.25);
                \draw [draw=none,fill=gray!30] (12.5,4.5) -- (14.5,2.5) -- (14.5,1.5) -- (11,1.5)
                    (12.5,-4.5) -- (14,-3) -- (12.5,-1.5) -- (11,-1.5)
                    (19.5,4.5) -- (17.5,2.5) -- (17.5,1.5) -- (21,1.5)
                    (19.5,-4.5) -- (18,-3) -- (19.5,-1.5) -- (21,-1.5);
                \draw [thick,fill=gray!30] (13,-1) -- (14.5,-2.5) -- (17.5,-2.5) -- (19,-1) -- (19,.5) -- (17.5,.5) -- (14.5,.5) -- (13,.5) -- cycle
                    (16.5,2.5) -- (16.5,1.5) -- (15.5,1.5) -- (15.5,2.5) -- cycle;
                \draw [thick] (14.5,2.5) -- (14.5,1.5) -- (12.5,1.5) -- (11,1.5)
                    (17.5,2.5) -- (17.5,1.5) -- (19.5,1.5) -- (21,1.5)
                    (14,-3) -- (12.5,-1.5) -- (11,-1.5) (18,-3) -- (19.5,-1.5) -- (21,-1.5);
                \draw [ultra thick] (12.5,4.5) -- (14.5,2.5) -- (17.5,2.5) -- (19.5,4.5)
                    (12.5,-4.5) -- (14.5,-2.5) -- (17.5,-2.5) -- (19.5,-4.5);
                    
                \draw node at (16,-1) {$v$};
                \draw node at (12.5,2.75) {$u_1$};
                \draw node at (19.5,2.75) {$u_3$};
                \draw node at (12.5,-2.75) {$u_2$};
                \draw node at (19.5,-2.75) {$u_4$};
                \draw node at (16,2) {$u_4$};
                \draw node at (12,0) {\LARGE$\circlearrowleft$};
                \draw node at (20,0) {\LARGE$\circlearrowright$};
                \draw node at (14.5,1) {\Large$\circlearrowleft$};
                \draw node at (17.5,1) {\Large$\circlearrowleft$};
            \end{tikzpicture}
        \caption{Two possible choices for making the link of the vertex disc of Figure \ref{fig:disc-lnk} connected.}
        \label{fig:2-chcs}
    \end{figure}
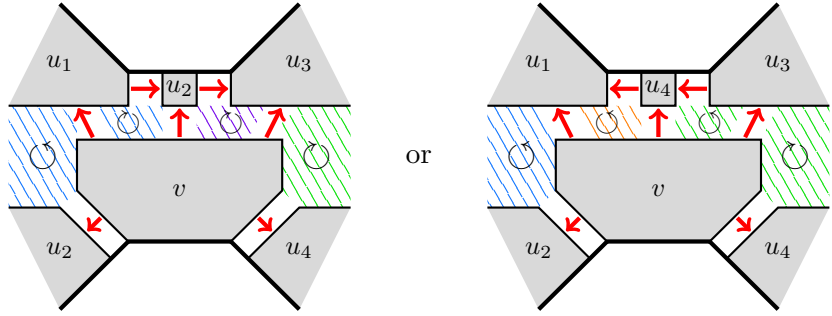
    
    We can now apply the following procedure: we alternately apply the operation of \S{}\ref{subsec:conn-lks} to make links connected --- and with the choice of only adding cellular discs of positive orientation --- and then the operation described above to remove folding.
    Each iteration of the latter removes one cellular disc of each orientation, and each iteration of the former does not increase the number of cellular discs of negative orientation.
    Hence, the total number of discs of negative orientation decreases strictly at each pair of iterations, ensuring that the process terminates in an admissible surface that is both non-folded and with connected links.
    
    Note that the operations just described do not impact the boundary map (up to homotopy), and hence the monotonicity, of $f$.
    Hence we can first apply Lemma \ref{lmm:monot} and ensure that the resulting surface is monotone.
    We can also readily remove any disc- or sphere-component.
    
    This proves the following:
    
    \begin{lmm}[Non-folded admissible surfaces]\label{lmm:or-nmix}
        Fix $\alpha\in H_2\left(S,c\right)$.
        Given a transverse incompressible admissible surface $f:(\Sigma,\partial\Sigma)\rightarrow(S,c)$ with $f_*[\Sigma]=n(\Sigma)\alpha$, there is a non-folded admissible surface $f':\left(\Sigma',\partial\Sigma'\right)\rightarrow\left(S,c\right)$ with connected links, with $f'_*\left[\Sigma'\right]=n\left(\Sigma'\right)\alpha$, and such that
        \[
            \frac{-\chi^-\left(\Sigma'\right)}{n\left(\Sigma'\right)}\leq\frac{-\chi^-(\Sigma)}{n(\Sigma)}.
        \]
        Moreover, $f'$ can be assumed to be monotone and disc- and sphere-free.\qed
    \end{lmm}

    \subsection{Asymptotic promotion to orientation-perfect surfaces}\label{subsec:or-perf}
    
    In order to obtain isometric embedding results for surface groups in \S{}\ref{sec:isom}, we will need admissible surfaces to satisfy the following orientation property:
    \begin{defi}
        Let $S$ be a cellulated, oriented, compact, connected surface. A transverse incompressible admissible surface $f:(\Sigma,\partial\Sigma)\rightarrow(S,\gamma)$ is \emph{orientation-perfect} if there are no two cellular discs in $\Sigma$ that map to the same $2$-cell of $S$ with opposite orientations.
    \end{defi}
    
    There is an operation which one may be very tempted to perform to obtain an orientation-perfect surface: given two cellular discs of $\Sigma$ mapping to the same $2$-cell $\sigma$ of $S$ with opposite orientations, one obtains a new admissible surface by removing the two cellular discs and gluing the two resulting boundary components to one another.
    This changes neither $\partial f$ nor $f_*[\Sigma]$; however, $-\chi^-(\Sigma)$ increases by $2$.
    If we do this carelessly, then we get a worse estimate of $\scl$ or $\gromnorm{\cdot}$.
    
    Instead, we will perform this operation in an asymptotic way that is inspired by Chen's \emph{asymptotic promotion} \cite{chen} --- albeit in a much simpler case.
    This comes at a cost: we won't be able to obtain extremal surfaces anymore; however, we will still be able to compute $\scl$ or $\gromnorm{\cdot}$.
    
    We start with a transverse, incompressible, monotone, disc- and sphere-free, non-folded admissible surface with connected links $f:(\Sigma,\partial\Sigma)\rightarrow(S,c)$ with $f_*[\Sigma]=n(\Sigma)\alpha$, and we assume that $f$ is not orientation-perfect.
    Fix a small $\varepsilon>0$, and pick a large $N\in\mathbb{N}_{\geq1}$ such that $\frac{1}{N}\leq\varepsilon$.
    Let ${\Sigma_0}\rightarrow\Sigma$ be a degree-$N$ covering under which the preimage of every connected component of $\Sigma$ is connected.
    The composite map ${\Sigma_0}\rightarrow\Sigma\rightarrow S$ is also a transverse, incompressible, monotone, disc- and sphere-free, non-folded admissible surface with connected links, with $\chi^-\left({\Sigma_0}\right)=N\chi^-(\Sigma)$ and $n\left({\Sigma_0}\right)=Nn\left(\Sigma\right)$.
    Since $f$ is non-folded but not orientation-perfect, there are two cellular discs in distinct components of ${\Sigma_0}$ that map to the same $2$-cell of $S$ with opposite orientations.
    We remove those two discs and glue the resulting boundary components to one another in a way that is compatible with the map $f$.
    There is an admissible surface $f':\left({\Sigma}_0',\partial{\Sigma}_0'\right)\rightarrow(S,c)$ resulting from this operation, which is still transverse, incompressible, monotone, disc- and sphere-free; it satisfies $f'_*\left[{\Sigma}_0'\right]=Nn\left(\Sigma\right)\alpha$, and
    \[
        -\chi^-\left({\Sigma}_0'\right)=-\chi^-\left({\Sigma_0}\right)+2=-N\chi^-(\Sigma)+2.
    \]
    Therefore
    \[
        \frac{-\chi^-\left({\Sigma}_0'\right)}{n\left({\Sigma}_0'\right)}\leq\frac{-\chi^-(\Sigma)}{n(\Sigma)}+2\varepsilon.
    \]
    We can then perform the process of \S{}\ref{subsec:or-nmix} again to ensure that ${\Sigma}_0'$ is non-folded and with connected links.
    
    After the complete operation, the number of connected components of $\Sigma$ has decreased by one, while the quantity $\frac{-\chi^-(\Sigma)}{n(\Sigma)}$ has not increased more than a controlled arbitrarily small amount.
    Since $\Sigma$ has a finite number of connected components, we may iterate until we obtain an orientation-perfect surface.
    We obtain the following:
    \begin{lmm}[Orientation-perfect admissible surfaces]\label{lmm:or-perf}
        Fix $\alpha\in H_2\left(S,c\right)$.
        Given an $\varepsilon>0$ and a transverse incompressible admissible surface $f:(\Sigma,\partial\Sigma)\rightarrow(S,c)$ with $f_*[\Sigma]=n(\Sigma)\alpha$, there is an orientation-perfect admissible surface $f':\left(\Sigma',\partial\Sigma'\right)\rightarrow\left(S,c\right)$, with $f'_*\left[\Sigma'\right]=n\left(\Sigma'\right)\alpha$, and such that
        \[
            \frac{-\chi^-\left(\Sigma'\right)}{n\left(\Sigma'\right)}\leq\frac{-\chi^-(\Sigma)}{n(\Sigma)}+\varepsilon.
        \]
        Moreover, $f'$ can be assumed to be monotone, disc- and sphere-free, non-folded, and with connected links.\qed
    \end{lmm}
    
    \begin{rmk}\label{rk:asymp-prom-not-nec}
        If the $1$-chain $c$ consists simply of an element $w\in\pi_1S$, and if $f:(\Sigma,\partial\Sigma)\rightarrow(S,w)$ is an admissible surface for $\scl(w)$ (not for $\gromnorm{\alpha}$), then we can bypass the asymptotic promotion argument and in fact replace $\Sigma$ with a connected admissible surface.
        Indeed, consider the connected components $\left\{\Sigma_i\right\}_i$ of $\Sigma$, and observe that the restriction of $f$ to each $\Sigma_i$ is an admissible surface for $\scl(w)$.
        (But note that distinct components may represent distinct classes in $H_2(S,w)$.)
        We have
        \[
            \frac{-\chi^-(\Sigma)}{n(\Sigma)}=\frac{\sum_i\left(-\chi^-\left(\Sigma_i\right)\right)}{\sum_in\left(\Sigma_i\right)}\geq\min_i\frac{-\chi^-\left(\Sigma_i\right)}{n\left(\Sigma_i\right)}.
        \]
        Hence there is a component $\Sigma_i$ of $\Sigma$ for which ${-\chi^-\left(\Sigma_i\right)}/{n\left(\Sigma_i\right)}\leq{-\chi^-(\Sigma)}/{n(\Sigma)}$, and we may replace $\Sigma$ with $\Sigma_i$.
        Now $\Sigma$ is connected, so making it non-folded is enough to guarantee that it is orientation-perfect.
    \end{rmk}
    
    \subsection{Standard form}
    
    We have shown the following:
    \begin{prop}[Standard form]\label{prop:std-form}
        Let $S$ be an oriented, compact, connected surface, let $c\in C_1\left(\pi_1S;\mathbb{Z}\right)$ be an integral chain, and $\alpha\in H_2\left(S,c;\mathbb{Q}\right)$. Then
        \begin{enumerate}
            \item The relative Gromov seminorm of $\alpha$ can be computed via
            \[
                \gromnorm{\alpha}=\inf_{f,\Sigma}\frac{-2\chi^-(\Sigma)}{n(\Sigma)},
            \]
            where the infimum is taken over all admissible surfaces $f:(\Sigma,\partial\Sigma)\rightarrow\left(S,c\right)$ that are transverse, incompressible, monotone, disc- and sphere-free, orientation-perfect, with connected links for some cellulation of $S$.
            
            Such an admissible surface is said to be in \emph{perfect standard form}.
            
            \item If there exists an extremal surface for $\gromnorm{\alpha}$ (i.e.\ realising the infimum in Definition \ref{defi:rel-grom}), then there exists one which is transverse, incompressible, monotone, disc- and sphere-free, non-folded, with connected links for some cellulation of $S$.
            
            Such an admissible surface is said to be in \emph{standard form}.
        \end{enumerate}
    \end{prop}
    \begin{proof}
        This follows from \S{}\ref{subsec:transv} (for transversality) and Lemmas \ref{lmm:monot} (for monotonicity), \ref{lmm:conn-lks} (for connected links), \ref{lmm:or-nmix} (for the non-folding property), and \ref{lmm:or-perf} (for the orientation-perfect property).
    \end{proof}
    
    \begin{rmk}
        By the discussion of \S{}\ref{subsec:or-nmix}, an orientation-perfect admissible surface is automatically non-folded.
        It follows that an admissible surface in perfect standard form is also in standard form.
    \end{rmk}
        
    It follows from Proposition \ref{prop:scl-inf-gt} that the obvious analogue of Proposition \ref{prop:std-form} holds for $\scl$: the stable commutator length of $c$ can be computed with surfaces in perfect standard form, and if there exists an extremal surface, then there exists one in standard form.
    
    Moreover, if the $1$-chain $c$ consists of a single element $w\in\pi_1S$, and if there exists an extremal surface for $\scl(w)$, then there exists one in perfect standard form (see Remark \ref{rk:asymp-prom-not-nec}).

\section{Isometries for scl and the relative Gromov seminorm}\label{sec:isom}
    
    We now have all the tools we need to prove our isometric embedding theorems.
    We consider $S$ an oriented, compact, connected surface, and $T\subseteq S$ a subsurface that is \emph{$\pi_1$-injective}, in the sense that the induced morphism
    \[
        \iota:\pi_1T\hookrightarrow\pi_1S
    \]
    is injective.
    We would like to understand when $\iota$ is a (strong) isometric embedding for $\scl$ or the relative Gromov seminorm.
    We will identify $\pi_1T$ with its image in $\pi_1S$.
    Hence, a chain $c\in C_1\left(\pi_1T;\mathbb{Z}\right)$ can also be seen as a chain in $C_1\left(\pi_1S;\mathbb{Z}\right)$, and admissible surfaces for $c$ can be considered either in $T$ or in $S$.
    
    \subsection{Main theorem}
    
    Our main technical result is the following, which says that, with our standard form and with appropriate homology vanishing conditions, an admissible surface in $S$ for a chain in $T$ is in fact entirely contained in $T$.
    
    \begin{thm}\label{thm:main}
        Let $S$ be a cellulated, oriented, compact, connected surface, let $T\subseteq S$ be a $\pi_1$-injective subcomplex, let $c\in C_1\left(\pi_1T;\mathbb{Z}\right)$ be an integral chain, and let $\alpha\in H_2(S,c;\mathbb{Q})$.
        Consider an admissible surface $f:(\Sigma,\partial\Sigma)\rightarrow(S,c)$ in $S$ with $f_*[\Sigma]=n(\Sigma)\alpha$ for some $n(\Sigma)\in\mathbb{N}_{\geq1}$, and with $f\left(\partial\Sigma\right)\subseteq T$.
        
        Let $R=\mathbb{Z}$ or $\mathbb{Q}$ or $\mathbb{R}$ and assume that one of the following holds:
        \begin{enumerate}
            \item $f$ is in standard form and $H_2(S,T;R)=0$, or\label{thm:main-1}
            \item $f$ is in perfect standard form and $f_*[\Sigma]=0$ in $H_2(S,T;R)$.\label{thm:main-2}
        \end{enumerate}
        Then $f(\Sigma)\subseteq T$.
    \end{thm}
    \begin{proof}
        Consider $S_0=\Imm f\subseteq S$, and let $T_0=S_0\cap T$.
        Hence, $S_0$ is a subcomplex of $S$, and $T_0$ is a ($\pi_1$-injective) subcomplex of $S_0$.
        The map $f$ induces $f_0:\Sigma\rightarrow S_0$.
        Our subcomplex stability lemmas imply that
        \begin{itemize}
            \item $S_0$ is an $R$-orientable $2$-complex with small links by Lemmas \ref{lmm:subcx-stab-lks} and \ref{lmm:subcx-stab-or};
            \item With assumption \ref{thm:main-1} (i.e. $H_2(S,T;R)=0$), we have $H_2\left(S_0,T_0;R\right)=0$ by Lemma \ref{lmm:subcx-stab-van-fund-class};
            \item With assumption \ref{thm:main-2} (i.e. $f_*[\Sigma]=0$ in $H_2(S,T;R)$), then ${f_0}_*[\Sigma]=0$ in $H_2\left(S_0,T_0;R\right)$ by Lemma \ref{lmm:subcx-stab-van-fund-class}.
        \end{itemize}
        
        \begin{claim}
            We have an inclusion $\partial S_0\subseteq f_0\left(\partial\Sigma\right)$ (where $\partial S_0$ should be understood in the sense of Definition \ref{defi:boundary}).
        \end{claim}
        \begin{proof}[Proof of the claim]
            Let $e$ be an edge of $\partial S_0$.
            By definition, $e\subseteq S_0=\Imm f$, so there is a $1$-handle $H$ in $\Sigma$ that maps to $e$.
            Now each end of the $1$-handle $H$ can be either incident to a cellular disc or to $\partial\Sigma$.
            If each end is incident to a cellular disc, then those two cellular discs must map to the same $2$-cell $\sigma$ of $S_0$, because $e$ is incident to only one $2$-cell as it lies on $\partial S_0$.
            In this case, the two cellular discs map to $\sigma$ with opposite orientations (as in Figure \ref{fig:cells-opp-or}), which contradicts the non-folding property --- which $f$ has since it is in standard form.
            Therefore, at least one end of $H$ must be incident to $\partial\Sigma$.
            This implies that $e\subseteq f_0\left(\partial\Sigma\right)$.
        \end{proof}
        
        By assumption, $f_0\left(\partial\Sigma\right)\subseteq T_0$.
        Hence we have
        \[
            \partial S_0\subseteq f_0\left(\partial\Sigma\right)\subseteq T_0\subseteq S_0.
        \]
        
        With assumption \ref{thm:main-1}, we have $H_2\left(S_0,T_0;R\right)=0$, so it follows immediately from Lemma \ref{lmm:or-h2-zero} that every $2$-cell of $S_0$ is contained in $T_0$.
        
        With assumption \ref{thm:main-2}, we have ${f_0}_*[\Sigma]=0$ in $H_2\left(S_0,T_0;R\right)$.
        Let $\sigma$ be a $2$-cell in $S_0=\Imm f$.
        Recall that $f_0$ is in perfect standard form, so it is orientation-perfect.
        This means that all cellular discs of $\Sigma$ mapping to $\sigma$ do so with the same orientation.
        Hence, the image ${f_0}_*[\Sigma]$ in $H_2\left(S_0,T_0;R\right)$ has a term in $\sigma$ with nonzero coefficient.
        But ${f_0}_*[\Sigma]=0$ in $H_2\left(S_0,T_0;R\right)$, so we must have $\sigma\subseteq T_0$.
        This shows that every $2$-cell in $S_0$ is contained in $T_0$.
        
        \begin{claim}\label{claim:every-cell-incid-2cell}
            Every $0$- or $1$-cell of $S_0$ is incident to a $2$-cell of $S_0$.
        \end{claim}
        \begin{proof}[Proof of the claim]
            Note first that there is no isolated $0$-cell in $S_0$ since $f_0$ is incompressible.
            Now assume for contradiction that there is a $1$-cell $e\subseteq S_0$ without any incident $2$-cell.
            Let $H$ be a $1$-handle of $\Sigma$ mapping to $e$.
            Then both ends of $H$ lie on $\partial\Sigma$.
            Hence any vertex disc in $\Sigma$ incident to $H$ meets $\partial\Sigma$ on both sides of $H$.
            But links of vertex discs are connected since $f$ is in standard form, so neither of the vertex discs incident to $H$ is incident to any other $1$-handle --- see Figure \ref{fig:isolated-handle}.
            \begin{figure}[htb]
                \centering
                    \centering
                    \begin{tikzpicture}[every node/.style={rectangle,draw=none,fill=none},scale=.5]
                        \draw [draw=none,fill=gray!30] (-3.5,2.5) -- (-.5,1.5) -- (-.5,-1.5) -- (-3.5,-2.5) -- cycle;
                        \draw [draw=none,fill=gray!30] (3.5,2.5) -- (.5,1.5) -- (.5,-1.5) -- (3.5,-2.5) -- cycle;
                        \draw [ultra thick,->,red,shorten <= 1pt,shorten >= 1pt] (-.5,0) -- (.5,0) node [above left,xshift=2pt] {{\small$e$}};
                        \draw [ultra thick] (-3.5,2.5) -- (-.5,1.5) -- (.5,1.5) -- (3.5,2.5)
                            (3.5,-2.5) -- (.5,-1.5) -- (-.5,-1.5) -- (-3.5,-2.5)
                            (-3.5,2.5) -- (-3.5,-2.5) (3.5,2.5) -- (3.5,-2.5);
                        \draw [thick] (-.5,1.5) -- (-.5,-1.5) (.5,1.5) -- (.5,-1.5);
                        
                        \draw node at (-2,0) {$u$};
                        \draw node at (2,0) {$v$};
                        \draw node at (0,-3.5) {$\Sigma$};
                        \draw node at (-1,2.5) {$\partial\Sigma$};
                        
                        \draw [<-,dashed] (-.25,-.75) to [bend left=20] (-4.5,-1) node [left] {$H$};
                        
                        \draw node [draw,circle,fill,inner sep=1.5pt,label=above:$u$] (u) at (7.75,0) {};
                        \draw node [draw,circle,fill,inner sep=1.5pt,label=above:$v$] (v) at (10.25,0) {};
                        \draw [ultra thick,red] (u) -- (v) coordinate [midway] (mid1);
                        \draw [ultra thick,red,->] (u) -- (mid1) node [above] {{\small$e$}};
                        
                        \draw node at (9,-1.5) {$S_0$};
                        
                        \draw [->] (4.75,0) -- (6.25,0) node [midway,above] {$f$};
                    \end{tikzpicture}
                \caption{A $1$-handle mapping to an edge with no incident $2$-cell.}
                \label{fig:isolated-handle}
            \end{figure}
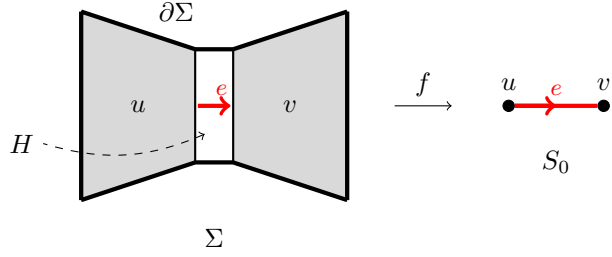
            Hence, $\Sigma$ has a disc component consisting of $H$ and its two incident vertex discs (note that these two discs are distinct by connectedness of their link).
            This is impossible since admissible surfaces in standard form are assumed to be disc- and sphere-free.
        \end{proof}

        Since every $2$-cell of $S_0$ is contained in $T_0$, it follows from Claim \ref{claim:every-cell-incid-2cell} that $S_0=T_0$, and therefore $\Imm f=S_0\subseteq T$ as wanted.
    \end{proof}
    
    \subsection{Isometric embeddings}\label{subsec:isom-emb}
    
    We now discuss applications of Theorem \ref{thm:main} to isometric embeddings for $\scl$ and $\gromnorm{\cdot}$.
    
    If $S$ is a surface, we say that a subsurface $T\subseteq S$ is \emph{$H_1$-injective} if the induced map $H_1(T)\rightarrow H_1(S)$ is injective.
    
    \begin{mthm}{A}[Isometric embedding for $\scl$]\label{thm:isom-scl}
        Let $S$ be an oriented, compact, connected surface with nonempty boundary, and let $T\subseteq S$ be a $H_1$-injective subsurface.
        Then the inclusion-induced map
        \[
            \iota:\pi_1T\hookrightarrow\pi_1S
        \]
        is a strong isometric embedding for $\scl$.
    \end{mthm}
    \begin{proof}
        Note first that $\pi_1T$ and $\pi_1S$ are free groups because $\partial S\neq\emptyset$.
        Since $\iota_*:H_1\left(\pi_1T\right)\rightarrow H_1\left(\pi_1S\right)$ is injective, we have $\rk\left(\pi_1T\right)=\rk\left(\Imm\iota\right)$, and it follows from the Hopf property for free groups that $\iota:\pi_1T\rightarrow\pi_1S$ is injective.
        
        Moreover, by the Universal Coefficient Theorem, the inclusion-induced map $H_1(T;\mathbb{Q})\rightarrow H_1(S;\mathbb{Q})$ is injective, or equivalently, the map
        \[
            \iota_*:H_1\left(\pi_1T;\mathbb{Q}\right)\rightarrow H_1\left(\pi_1S;\mathbb{Q}\right)
        \]
        is injective.
        
        It remains to show that $\iota$ preserves the stable commutator length of integral boundaries (strong isometry will follow from Proposition \ref{prop:isom-strong-isom}).
        
        Let $c\in B_1\left(\pi_1T;\mathbb{Z}\right)$.
        By Proposition \ref{prop:std-form}, $\scl_{\pi_1S}(c)$ can be computed as an infimum over all admissible surfaces $f:(\Sigma,\partial\Sigma)\rightarrow(S,c)$ in standard form.
        But since $S$ has nonempty boundary, $H_2(S)=0$, so injectivity of $H_1(T)\rightarrow H_1(S)$ implies that $H_2(S,T)=0$ by the long exact sequence of $(S,T)$.
        Therefore, by Theorem \ref{thm:main}, every admissible surface in standard form in $S$ can be homotoped to an admissible surface in $T$.
        It follows that
        \[
            \scl_{\pi_1T}(c)\leq\scl_{\pi_1S}(c),
        \]
        and the reverse inequality always holds by monotonicity of $\scl$ (see Proposition \ref{prop:monotonicity-scl}).
    \end{proof}
    
    In Theorem \ref{thm:isom-scl}, observe that $H_1$-injectivity of $T$ could be replaced with the equivalent assumption that $H_2(S,T)=0$.
    With that assumption, the theorem would also hold when $S$ is closed since in that case, $H_2(S,T)=0$ implies that $S=T$ by Corollary \ref{cor:rel-hom-srf}.
    However, this would only add a trivial statement and therefore the generality of the theorem would not be increased.
    
    Our second isometric embedding theorem is the following, which also applies --- and gives a nontrivial result --- in the closed case:
    
    \begin{mthm}{B}[Isometric embedding for the relative Gromov seminorm]\label{thm:isom-gromnorm}
        Let $S$ be an oriented, compact, connected surface, let $T\subseteq S$ be a $\pi_1$-injective subsurface, and let $c\in C_1\left(\pi_1T;\mathbb{Z}\right)$ be an integral chain in $T$.
        Then the inclusion-induced map
        \[
            \iota:H_2(T,c;\mathbb{Q})\hookrightarrow H_2(S,c;\mathbb{Q})
        \]
        is an injective isometric embedding for $\gromnorm{\cdot}$.
    \end{mthm}
    \begin{proof}
        Note that injectivity follows from the long exact sequence of the triple $(S,T,c)$ (see Proposition \ref{prop:les-rel-h}) since $H_3(S,T)=0$.
    
        Let $\alpha\in H_2(T,c;\mathbb{Q})$.
        Then Proposition \ref{prop:std-form} says that $\gromnorm{\iota\alpha}$ can be computed as an infimum over all admissible surfaces $f:(\Sigma,\partial\Sigma)\rightarrow(S,c)$ in perfect standard form.
        Let $f$ be such an admissible surface, with $f_*[\Sigma]=n(\Sigma)\iota\alpha$ in $H_2(S,c;\mathbb{Q})$.
        By the long exact sequence of the triple $(S,T,c)$ (see Proposition \ref{prop:les-rel-h}), $\iota\alpha$ maps to zero in $H_2(S,T;\mathbb{Q})$, and so the image of $f_*[\Sigma]$ in $H_2(S,T;\mathbb{Q})$ is also zero.
        Therefore, Theorem \ref{thm:main} applies and $f$ can be homotoped to an admissible surface in $T$.
        This proves that $\gromnorm{\alpha}\leq\gromnorm{\iota\alpha}$, and the reverse inequality always holds since $\gromnorm{\cdot}$ is monotone with respect to continuous maps.
    \end{proof}
    
    In fact, applying Theorem \ref{thm:isom-gromnorm} to the context of surfaces with nonempty boundary yields a stronger version of Theorem \ref{thm:isom-scl}:
    
    \begin{mcor}{C}\label{cor:isom-scl}
        Let $S$ be an oriented, compact, connected surface with nonempty boundary and let $T\subseteq S$ be a $\pi_1$-injective subsurface.
        Then the inclusion-induced map
        \[
            \iota:\pi_1T\hookrightarrow\pi_1S
        \]
        is an isometric embedding for $\scl$.
    \end{mcor}
    \begin{proof}
        The morphism $\iota$ is assumed to be injective, so it suffices to prove that it preserves the $\scl$ of homologically trivial $1$-chains.
        Let $c\in B_1\left(\pi_1T;\mathbb{Z}\right)$.
        Since $H_2(T)=0$, the long exact sequence of the pair $(T,c)$ (see Proposition \ref{prop:les-rel-h}) shows that there is a unique class $\alpha\in H_2\left(T,c;\mathbb{Q}\right)$ such that $\partial\alpha=\left[\coprod_iS^1\right]$.
        By Proposition \ref{prop:scl-inf-gt}, we have $4\scl_{\pi_1T}(c)=\gromnorm{\alpha}$.
        Similarly, $H_2(S)=0$ and $4\scl_{\pi_1S}(c)=\gromnorm{\iota\alpha}$.
        But Theorem \ref{thm:isom-gromnorm} implies that $\gromnorm{\alpha}=\gromnorm{\iota\alpha}$, so $\scl_{\pi_1T}(c)=\scl_{\pi_1S}(c)$.
    \end{proof}

\section{Extremal surfaces and quasimorphisms}\label{sec:extr}

    We conclude with a discussion of how our isometric embeddings of surfaces behave with respect to extremal surfaces and quasimorphisms.
    To be more precise, consider $\iota:G\hookrightarrow H$ an isometric embedding for $\scl$.
    Note in particular that $\iota$ is injective and induces an embedding $K(G,1)\hookrightarrow K(H,1)$.
    Given a $1$-chain $c\in C_1(G;\mathbb{Z})\hookrightarrow C_1(H;\mathbb{Z})$, our aim is to find conditions under which
    \begin{itemize}
        \item There is an extremal surface $\left(\Sigma,\partial\Sigma\right)\rightarrow\left(K(G,1),c\right)$ for $\scl_G(c)$ that is also extremal for $\scl_H(c)$, or
        \item There is an extremal quasimorphism $\varphi\in Q(H)$ for $\scl_H(c)$ that restricts to an extremal quasimorphism $\varphi_{\left|G\right.}\in Q(G)$ for $\scl_G(c)$.
    \end{itemize}
    We address these problems for the isometric embedding $\iota:\pi_1T\hookrightarrow\pi_1S$ of Theorem \ref{thm:isom-scl}.

    \subsection{Extremal surfaces}
    
    Recall that an \emph{extremal surface} for $\scl_{G}(c)$ is one that realises the infimum in Proposition \ref{prop:scl-topo}.
    A major result, due to Calegari \cite{cal-sclrat}, is that extremal surfaces exist for all $c\in B_1(G;\mathbb{Z})$ if $G$ is a free group.
    
    Proposition \ref{prop:std-form} says that, for the purpose of finding extremal surfaces, we can assume that admissible surfaces are in standard form --- but not necessarily in perfect standard form.
    In the context of Theorem \ref{thm:isom-scl}, this is sufficient: $H_1$-injectivity implies that $H_2(S,T)=0$ since $H_2(S)=0$, so Theorem \ref{thm:main} with assumption \ref{thm:main-1} says that any admissible surface $(\Sigma,\partial\Sigma)\rightarrow(S,c)$ in standard form is in fact contained in $T$.
    Note also that $\pi_1S$ and $\pi_1T$ are free groups, so extremal surfaces exist by Calegari's Theorem \cite{cal-sclrat}.
    This gives the following:
    
    \begin{cor}\label{cor:extr-srf}
        Let $S$ be an oriented, compact, connected surface with nonempty boundary, and let $T\subseteq S$ be a $H_1$-injective subsurface.
        Let $c\in B_1\left(\pi_1T;\mathbb{Z}\right)$.
        Then there exists an admissible surface $f:(\Sigma,\partial\Sigma)\rightarrow(T,c)$ that is extremal for both $\scl_{\pi_1T}(c)$ and $\scl_{\pi_1S}(c)$.\qed
    \end{cor}
    
    Note however that, even if extremal surfaces were known to exist for the relative Gromov seminorm, we would not obtain an analogue of Corollary \ref{cor:extr-srf} in that setting.
    Indeed, to prove Theorem \ref{thm:isom-gromnorm}, we needed to apply Theorem \ref{thm:main} with assumption \ref{thm:main-2} and work with admissible surfaces in perfect standard form.
    But an asymptotic promotion argument was necessary to obtain the perfect standard form (see \S{}\ref{subsec:or-perf}), and this does not preserve extremal surfaces.
    
    In the case where $S$ has nonempty boundary and the subsurface $T\subseteq S$ is only $\pi_1$-injective rather than $H_1$-injective, then Corollary \ref{cor:isom-scl} says that $\pi_1T\hookrightarrow\pi_1S$ is still an isometric embedding for $\scl$.
    In general, as Corollary \ref{cor:isom-scl} relies on Theorem \ref{thm:isom-gromnorm}, this isometric embedding might not preserve extremal surfaces, but in the special case where the $1$-chain $c$ consists of a single element $w$ of $\pi_1T$, then Remark \ref{rk:asymp-prom-not-nec} says that the asymptotic promotion argument can be bypassed, and therefore extremal surfaces can be assumed to be in perfect standard form.
    We obtain:
    
    \begin{cor}\label{cor:extr-srf-single-element}
        Let $S$ be an oriented, compact, connected surface with nonempty boundary, and let $T\subseteq S$ be a $\pi_1$-injective subsurface.
        Let $w\in\left[\pi_1T,\pi_1T\right]$.
        Then there exists an admissible surface $f:(\Sigma,\partial\Sigma)\rightarrow(T,w)$ that is extremal for both $\scl_{\pi_1T}(w)$ and $\scl_{\pi_1S}(w)$.\qed
    \end{cor}

    \subsection{Extremal quasimorphisms}
    
    Recall that a \emph{quasimorphism} on a group $G$ is a map $\phi:G\rightarrow\mathbb{R}$ such that
    \[
        \sup_{a,b\in G}\left|\phi(ab)-\phi(a)-\phi(b)\right|<\infty.
    \]
    This supremum is called the \emph{defect} of $\phi$ and denoted by $D(\phi)$.
    We say that $\phi$ is \emph{homogeneous} if $\phi\left(w^n\right)=n\phi(w)$ for all $w\in G$ and $n\in\mathbb{Z}$.
    We denote by $Q(G)$ the space of homogeneous quasimorphisms $G\rightarrow\mathbb{R}$.
    
    Given a quasimorphism $\phi:G\rightarrow\mathbb{R}$, we can naturally extend $\phi$ to a map $C_1\left(G;\mathbb{R}\right)\rightarrow\mathbb{R}$ by linearity.
    
    The connection between quasimorphisms and $\scl$ is given by the following result, which says essentially that $\left(Q(G),D(\cdot)\right)$ is the dual space of $\left(B_1(G;\mathbb{R}),\scl_G\right)$ (after quotienting by the kernels of the respective seminorms):
    
    \begin{prop}[Bavard Duality \cite{bavard}]\label{prop:bavard}
        Let $G$ be a group and $c\in C_1\left(G;\mathbb{R}\right)$ be a $1$-chain. Then
        \[
            \scl_G(c)=\sup_{\substack{\phi\in Q(G)\\D(\phi)\neq0}}\frac{\phi(c)}{2D(\phi)}.
        \]
    \end{prop}
    
    A quasimorphism is called \emph{extremal} if it realises the supremum in Proposition \ref{prop:bavard}.
    As opposed to extremal surfaces, extremal quasimorphisms exist for all $1$-boundaries \cite{cal-scl}*{Proposition 2.88}, but finding an explicit extremal quasimorphism for a given element is usually a hard problem.
    There are however some results of this form; the following will be of particular interest to us:
    
    \begin{prop}[Calegari \cite{cal-fnb}]\label{prop:rot-extr}
        Let $S$ be a hyperbolic, compact, connected surface with $\partial S\neq\emptyset$. Let $c\in B_1\left(\pi_1S;\mathbb{Z}\right)$.
        Then the following are equivalent:
        \begin{enumerate}
            \item There is an admissible surface $f:(\Sigma,\partial\Sigma)\rightarrow(S,c)$ for $\scl_{\pi_1S}(c)$ that is immersed and orientation-preserving --- we say that $c$ \emph{rationally bounds a positive immersed surface}.
            \item The rotation quasimorphism $\rot_S$ is extremal for $c$.
        \end{enumerate}
    \end{prop}
    
    The \emph{rotation quasimorphism} is an object that encodes the dynamics of the action of $\pi_1S$ on the boundary of the hyperbolic plane given by the choice of a hyperbolic structure.
    The defect of $\rot_S$ is always $1$.
    See \cite{cal-fnb} or \cite{cal-scl}*{\S{}2.3.3} for more details.
    Note that Proposition \ref{prop:rot-extr} applies in particular to the $1$-chain $c$ given by the (oriented) boundary of $S$.
    
    We make the following observation:
    
    \begin{prop}\label{prop:restr-rot}
        Let $S$ be a hyperbolic, compact, connected surface with $\partial S\neq\emptyset$, and let $T\subseteq S$ be a $H_1$-injective convex subsurface.
        Given $c\in B_1\left(\pi_1T;\mathbb{Z}\right)$, there is an equality
        \[
            \rot_T(c)=\rot_S(c).
        \]
    \end{prop}
    \begin{proof}
        There is an interpretation of the rotation number in terms of the area enclosed by a chain: $\rot_S(c)=\frac{1}{2\pi}\area_S(c)$.
        We refer to \cite{cal-scl}*{Lemma 4.68} for full details, but it suffices for our purpose to say that $\area_S(c)=\sum_in_i\area\left(\sigma_i\right)$, where $b=\sum_in_i\sigma_i$ is a cellular $2$-chain in $S$ with $db=c$.
        Now $H_1$-injectivity of $T$ implies that $H_2(S,T)=0$ since $H_2(S)=0$, so the differential
        \[
            d:C^\textnormal{cell}_2(S,T)\rightarrow C^\textnormal{cell}_1(S,T)
        \]
        is injective.
        Hence, the fact that $db=c\in C^\textnormal{cell}_1(T)$ implies that $b\in C^\textnormal{cell}_2(T)$, and therefore $\area_T(c)=\sum_in_i\area\left(\sigma_i\right)=\area_S(c)$.
    \end{proof}
    
    It follows that, if $T$ is $H_1$-injective, then for any chain $c\in B_1(\pi_1T;\mathbb{Z})$ that rationally bounds a positive immersed surface in $T$, we have
    \[
        \scl_{\pi_1T}(c)=\frac{1}{2}\rot_T(c)=\frac{1}{2}\rot_S(c)\leq\scl_{\pi_1S}(c)\leq\scl_{\pi_1T}(c),
    \]
    so $\rot_S$ is an extremal quasimorphism for $c$ in $S$ and restricts to $\rot_T$, which is an extremal quasimorphism for $c$ in $T$.
    
    It is natural at this point to ask whether it is equivalent for a chain $c\in B_1(\pi_1T;\mathbb{Z})$ to rationally bound a positive immersed surface in $T$ or in $S$.
    Using the same kind of argument as in the proof of Proposition \ref{prop:restr-rot}, it is easy to see that this is true if the word `immersed' is replaced with `embedded'.
    In general, using Scott's theorem on subgroup separability of surface groups \cites{scott1,scott2}, we can lift an immersed surface to an embedded one, and thus obtain an affirmative answer:
    
    \begin{prop}\label{prop:rat-bnd-t-s}
        Let $S$ be a hyperbolic, compact, connected surface with $\partial S\neq\emptyset$, and let $T\subseteq S$ be a $H_1$-injective convex subsurface.
        Given $c\in B_1\left(\pi_1T;\mathbb{Z}\right)$, the following are equivalent:
        \begin{enumerate}
            \item $\rot_S$ is an extremal quasimorphism for $\scl_{\pi_1S}(c)$.\label{prop:rat-bnd-t-s-1}
            \item $c$ rationally bounds a positive immersed surface in $S$\label{prop:rat-bnd-t-s-2}.
            \item $\rot_T$ is an extremal quasimorphism for $\scl_{\pi_1T}(c)$\label{prop:rat-bnd-t-s-3}.
            \item $c$ rationally bounds a positive immersed surface in $T$\label{prop:rat-bnd-t-s-4}.
        \end{enumerate}
    \end{prop}
    \begin{proof}
        The equivalences $\ref{prop:rat-bnd-t-s-1}\Leftrightarrow\ref{prop:rat-bnd-t-s-2}$ and $\ref{prop:rat-bnd-t-s-3}\Leftrightarrow\ref{prop:rat-bnd-t-s-4}$ follow from Proposition \ref{prop:rot-extr}.
        It is clear that $\ref{prop:rat-bnd-t-s-4}\Rightarrow\ref{prop:rat-bnd-t-s-2}$ since $T\subseteq S$, so it remains to prove that $\ref{prop:rat-bnd-t-s-2}\Rightarrow\ref{prop:rat-bnd-t-s-4}$.
        
        Assume that \ref{prop:rat-bnd-t-s-2} holds: there is an admissible surface $f:\left(\Sigma,\partial\Sigma\right)\rightarrow(S,c)$ for $\scl_{\pi_1S}(c)$ that is immersed and orientation-preserving.
        It follows from Scott's Theorem \cites{scott1,scott2} that there is a finite covering $g:S_0\rightarrow S$ over which $f$ lifts as an embedding $f_0:\Sigma\hookrightarrow S_0$:
        \[
            \vcenter{\hbox{\begin{tikzpicture}[every node/.style={draw=none,fill=none,rectangle}]
                \node (B) at (2,1.5) {$S_0$};
                \node (Ap) at (0,0) {$\Sigma$};
                \node (Bp) at (2,0) {$S$};
                
                \draw [->] (Ap) -> (B) node [midway,above left] {$f_0$};
                \draw [->] (Ap) -> (Bp) node [midway,above] {$f$};
                \draw [->] (B) -> (Bp) node [midway,right] {$g$};
            \end{tikzpicture}}}
        \]
        Consider $T_0=g^{-1}\left(T\right)\subseteq S_0$.
        Since $T$ is $H_1$-injective, it follows that $H_2(S,T)=0$.
        This means that every connected component of $\overline{S\smallsetminus T}$ contains at least one boundary component of $S$.
        This property lifts to a finite cover, so $H_2\left(S_0,T_0\right)=0$.
        
        Now we have an embedded surface $f_0:\Sigma\hookrightarrow S_0$.
        For an appropriate cellular structure on $S_0$, this embedding gives rise to a cellular $2$-chain $b\in C_2^\textnormal{cell}\left(S_0\right)$ with $db\in C_1^\textnormal{cell}\left(T_0\right)$.
        The same argument as in the proof of Proposition \ref{prop:restr-rot} then gives $b\in C_2^\textnormal{cell}\left(T_0\right)$.
        As $f_0$ is an embedding, this implies that $f_0(\Sigma)\subseteq T_0$, so $f(\Sigma)\subseteq T$ and $c$ rationally bounds a positive immersed surface in $T$.
    \end{proof}
    
    \begin{rmk}
        We could also have used Theorem \ref{thm:isom-scl} to prove that $\ref{prop:rat-bnd-t-s-1}\Rightarrow\ref{prop:rat-bnd-t-s-3}$ in Proposition \ref{prop:rat-bnd-t-s}.
        Indeed, assume that $\rot_S$ is extremal for $\scl_{\pi_1S}(c)$.
        Recall that $\pi_1T\hookrightarrow\pi_1S$ is isometric and that $\rot_S$ and $\rot_T$ agree on $\pi_1T$ (by Proposition \ref{prop:restr-rot}).
        Therefore
        \[
            \scl_{\pi_1T}(c)=\scl_{\pi_1S}(c)=\frac{\rot_S(c)}{2}=\frac{\rot_T(c)}{2},
        \]
        so that $\rot_T$ is extremal for $\scl_{\pi_1T}(c)$.
        
        However, the previous proof, using Scott's Theorem, has the advantage of being independent of Theorem \ref{thm:isom-scl}.
    \end{rmk}
    
    This does not say whether the isometric embedding of Theorem \ref{thm:isom-scl} respects extremal quasimorphisms for all $1$-chains $c$, but it does for some of them:
    
    \begin{cor}
        Let $S$ be a hyperbolic, compact, connected surface with $\partial S\neq\emptyset$, and let $T\subseteq S$ be a $H_1$-injective convex subsurface.
        Let $c\in B_1\left(\pi_1T;\mathbb{Z}\right)$.
        
        If $c$ rationally bounds a positive immersed surface in $S$, then the rotation quasimorphism $\rot_S\in Q\left(\pi_1S\right)$ is extremal for $\scl_{\pi_1S}(c)$, and restricts to $\rot_T\in Q\left(\pi_1T\right)$ which is extremal for $\scl_{\pi_1T}(c)$.\qed
    \end{cor}
    
    This gives an alternative proof that the embedding $\pi_1T\hookrightarrow\pi_1S$ preserves the stable commutator length of every $c\in B_1\left(\pi_1T;\mathbb{Z}\right)$ which rationally bounds a positive immersed surface in $S$, independently of Theorem \ref{thm:isom-scl}.

\bibliography{PhD-Refs.bib}

\end{document}